\documentclass[10pt]{amsart}
\usepackage[latin1]{inputenc}
\usepackage{amsthm}
\usepackage{amssymb}
\usepackage{amsmath}
\usepackage{mathtools}
\usepackage{multicol}
\numberwithin{equation}{section} %% Comment out for sequentially-numbered
\theoremstyle{plain}
\newtheorem{thm}{Theorem}
\newtheorem*{thmA}{Theorem A}

\newtheorem*{thmB}{Theorem B}
\newtheorem*{thmC}{Theorem C}
\newtheorem*{propD}{Proposition D}
\newtheorem*{thmE}{Theorem E}
\newtheorem*{thmF}{Theorem F}

\newtheorem{prop}[thm]{Proposition}
\newtheorem{conj}[thm]{Conjecture}

\newtheorem{cor}[thm]{Corollary}
\theoremstyle{definition}
\newtheorem{defn}[thm]{Definition}
\newtheorem{ass}[thm]{Standing Hypothesis}
\newtheorem{rem}[thm]{Remark}

\usepackage{graphicx}
\usepackage{amscd}
\usepackage[all,knot,poly]{xy}
\usepackage[english]{babel}
\usepackage{mathrsfs}
\usepackage{hyperref}
\usepackage{xcolor}
\newcommand{\C}{\mathbb{C}}
\newcommand{\R}{\mathbb{R}}
\newcommand{\Q}{\mathbb{Q}}
\newcommand{\Z}{\mathbb{Z}}
\newcommand{\N}{\mathbb{N}}

\newcommand{\F}{\mathbb{F}}

\newcommand{\id}{\trianglelefteq}

\newcommand{\Tor}{\mathrm{Tor}}

\newcommand{\Span}{\mathrm{span}}

\newcommand{\gen}[1]{\langle #1\rangle}

\renewcommand{\ker}{\mathrm{Ker}}

\newcommand{\im}{\mathrm{Im}}

\newcommand{\Aut}{\mathrm{Aut}}

\newcommand{\uK}{universally Koszul }
\newcommand{\squig}{\mathbin{\rotatebox[origin=c]{-90}{$\xi$}}}
\newcommand{\pr}{\circledast}
\newcommand{\gr}{\mathrm{gr}\,}
\newcommand{\chr}{\mathrm{char}\,}
\newcommand{\heart}{\ensuremath\heartsuit}

\author{J\'an Min\'a\v{c}, Marina Palaisti, Federico W. Pasini, Nguy$\tilde{\text{\^{e}}}$n Duy T\^{a}n}
\title{Enhanced Koszul properties in Galois cohomology}
\date{\today}
\dedicatory{Dedicated to David Eisenbud, with admiration and gratitude.}

\begin{document}
\begin{abstract}
We prove that Galois cohomology satisfies several surprisingly strong versions of Koszul properties, under a well known conjecture, in the finitely generated case. In fact, these versions of Koszulity hold for all finitely generated maximal pro-$p$ quotients of absolute Galois groups which are currently understood. We point out several of these unconditional results which follow from our work. We show how these enhanced versions are preserved under certain natural operations on algebras, generalising several results that were previously established only in the commutative case. The subject matter in this paper contains topics which are used in various branches of algebra and computer science.
\end{abstract}
\vspace*{-20pt}
\maketitle
\vspace*{-20pt}
\tableofcontents
\section{Introduction}
\subsection{Motivation}
An important outstanding problem in Galois theory is to distinguish absolute Galois groups of fields among all profinite groups. The first significant results in this direction were obtained by E.~Artin and O.~Schreier in 1927 in \cite{artsch1,artsch2}, where they show that any finite subgroup of an absolute Galois group has order $\leq 2$.
In 2011, M.~Rost and V.~Voevodsky completed the proof of the Bloch-Kato conjecture (see \cite{suslin:norm, haeseweibel, voevodsky:motivic, voevodsky:BK}), which describes the Galois cohomology of a field in terms of generators and relations, provided the field contains enough roots of unity. Recently, there has been some considerable activity related to the vanishing of higher Massey products in Galois cohomology and to the defining relations of Galois groups (see \cite{wick1}, \cite{hw}, \cite{mintan1}, \cite{efrmat}, \cite{mintan2}, \cite{wick2}, \cite{gumito}, \cite{matzri:massey}, \cite{efrqua}, \cite{harwit}, \cite{guimin}, \cite{mirota}).
Nevertheless, a complete classification of absolute Galois groups seems to be an overly difficult goal at the moment.
One may ask the simpler question of how to classify all maximal pro-$p$ quotients of absolute Galois groups for a given prime number $p$. These pro-$p$ groups carry information about all Galois extensions of degree a power of $p$ of a given field. Some recent results about these extensions may be found in \cite{mich1}, \cite{mich2}, \cite{andy}, \cite{chmish}. However,  this classification turns out to be a difficult open problem as well. 

In our paper we approach the quest of understanding Galois cohomology by resolving its relations through basic methods of homological algebra that date back to D.~Hilbert \cite{hilbert}. In other words we consider relations between relations, relations between relations between relations, and so on. Following Hilbert, we call these higher order relations \emph{syzygies}.

In \cite{priddy}, S.B.~Priddy singled out a class of graded algebras over fields with remarkable syzygies, which he named \emph{Koszul algebras}. There is a fascinating dichotomy involving them: namely, the Koszul property is very restrictive, yet Koszul algebras surprisingly appear in numerous exciting branches of representation theory, algebraic geometry, combinatorics, and computer algebra.

Koszul algebras are \emph{quadratic}, that is they are generated by degree $1$ elements, and relations are generated by degree $2$ elements. A standard reference on quadratic algebras is \cite{polpos}. An important consequence of Bloch-Kato conjecture is that Galois cohomology with $\F_p$ coefficients is a quadratic algebra, provided that the ground field contains a primitive $p$\textsuperscript{th} root of unity. Therefore L.~Positselski in his PhD thesis \cite{posit:thesis} (see also \cite{positselsky:koszul}) and Positselski and A.~Vishik in \cite{posivis:koszul} began to consider the idea that mod-$p$ Milnor K-theory and the Galois cohomology of a field as above may be Koszul algebras.

Priddy also showed in \cite{priddy} that a quadratic algebra is Koszul if and only if a certain specific and easily understandable complex is acyclic. In this case, that complex is actually a resolution of the relations of the algebra, thereby providing an immediate solution to our initial question. Taking into account that all relations in Galois cohomology are consequences of the simple equations $a+b=1$ for $a,b$ in the given ground field, we see that when Galois cohomology is indeed Koszul, it reveals itself to be as clear and as harmonious as a Mozart symphony.

Formally, Positselski conjectured the Koszulity of the Galois cohomology with $\F_p$ coefficients and the mod-$p$ Milnor K-theory of all fields containing a primitive $p$\textsuperscript{th} root of unity, and he proved this assertion for local and global fields, in 2014 (see \cite{positselsky:koszul2}). On the other hand, from the year 2000 onward, various stronger versions of Koszulity have been proposed and investigated in commutative algebra, most of the time in relation to algebraic geometry (see \cite{conca:u-Koszul}, \cite{hehire}, \cite{corova}, \cite{cotrva}, \cite{codero}, \cite{enhehi}), and extended to the noncommutative setting by D.I.~Piontkovski\u\i\ in \cite{piontk}. A standard reference on commutative algebra is \cite{eisenbud}.

In general, deciding whether a given quadratic algebra is Koszul is already difficult, because of the lack of any combinatorial criteria and screening tests (the failure of the ``natural candidate criteria'' is well expressed in \cite{posit:hilb}, \cite{roos} and \cite{piontk:hilb}). In addition, some of these enhanced versions of Koszulity are very restrictive even among Koszul algebras, hence one expects that algebras with these properties should be extremely rare. Paradoxically, however, proving or disproving these enhanced versions may be easier than deciding about the original Koszulity. Indeed this is true in our case. In this paper we show that, surprisingly, if a certain well-known conjecture in Galois theory is true (see Section \ref{sec:ET groups}), then nearly all of these enhanced forms of Koszulity hold for the Galois cohomology of those maximal pro-$p$ quotients of absolute Galois groups that are finitely generated. In fact our results encompass all finitely generated pro-$p$ groups that are currently understood to be maximal pro-$p$ quotients of absolute Galois groups.

\subsection{Notation and results}
Throughout this paper, $p$ is a prime number, $F$ is a field, $G_F=\mathrm{Gal}(F_{sep}/F)$ is the absolute Galois group of $F$ and $G_F(p)=\mathrm{Gal}(F(p)/F)$ is its maximal pro-$p$ quotient. Unless otherwise stated, we always assume the following.% Let $p$ be a prime number and let $F$ be a field satisfying the
\begin{ass}\label{ass}\
\begin{enumerate}
\item We work in the category of profinite groups, therefore all groups are understood to be topological groups. In particular, all homomorphisms are tacitly assumed to be continuous, and all subgroups to be closed. Moreover, a group $G$ is said to be finitely generated if it is \emph{topologically} finitely generated; that is, there are finitely many elements of $G$ such that the smallest (closed) subgroup of $G$ containing all of these elements is $G$ itself.
\item The field $F$ contains a primitive $p$\textsuperscript{th} root of unity.
\item The group $G_F(p)$ is finitely generated.
\end{enumerate}
%The field $F$ contains a prinitive $p$\textsuperscript{th} root of unity.
\end{ass}

Further, we denote $H^\bullet(G,\F_p)$ to be the group cohomology with continuous cochains of a profinite group $G$ with coefficients in $\F_p$ (see \cite[Chapter 1]{NSW}, \cite{serre}). Milnor's paper \cite{milnor} introduces a quadratic algebra which is now called Milnor K-theory, denoted $K_\bullet(F)$. It is the quotient of the tensor algebra on the multiplicative group $F^\times$ by the two-sided ideal generated by the tensors $a\otimes(1-a)$, $a\in F\setminus\{0,1\}$. An important special case of the Bloch-Kato conjecture claims that there is a natural graded isomorphism, the \emph{norm-residue map}
\[
h_\bullet:\frac{K_\bullet(F)}{pK_\bullet(F)}\to H^\bullet(G_F,\F_p).
\]
A proof of this conjecture was completed by Rost and Voevodsky in 2011 (see \cite{suslin:norm, haeseweibel, voevodsky:motivic, voevodsky:BK}). Two relevant corollaries are that $H^\bullet(G_F,\F_p)$ is a quadratic algebra, and that the inflation map
\[
\mathrm{inf}:H^\bullet(G_F(p),\F_p)\to H^\bullet(G_F,\F_p)
\]
is an isomorphism.

For the algebras $H^\bullet(G_F,\F_p)$, or equivalently $H^\bullet(G_F(p),\F_p)$, we investigate some properties related to Koszulity: original Koszulity, the existence of Koszul flags, the existence of Koszul filtrations, universal Koszulity, strong Koszulity, and the PBW property. The definitions of these properties are given in Section \ref{sec:koszul}, along with  relevant bibliographical information. The known relations between them are summarised in the following diagram, in which all implications are known to be not invertible.
\begin{center}
\[
\xymatrix@R-4pt{
\mbox{Universally Koszul}\ar@{}[dr]^(.25){}="aa"^(.75){}="bb" \ar@{=>} "aa";"bb" && \mbox{Strongly Koszul}\ar@{}[dl]^(.25){}="a1"^(.75){}="b1" \ar@{=>} "a1";"b1"&\\
& \mbox{Koszul filtration}\ar@{}[d]^(.25){}="a5"^(.75){}="b5" \ar@{=>} "a5";"b5" &&\mbox{PBW}\ar@{}[dll]^(.15){}="a2"^(.85){}="b2" \ar@{=>} "a2";"b2"\\
&\mbox{Koszul flag}\ar@{}[d]^(.25){}="a"^(.75){}="b" \ar@{=>} "a";"b" & &\\
& \mbox{Koszul}&&
}
\]
\end{center}
% \begin{center}
% \[
% \xymatrix@R-4pt{
% &&\mbox{Universally Koszul}\ar@{}[dr]^(.25){}="aa"^(.75){}="bb" \ar@{=>} "aa";"bb"\ar@{}[rr]^(.25){}="a3"^(.75){}="b3" \ar@{=>}|{/} "a3";"b3" && \mbox{Strongly Koszul}\ar@{}[dl]^(.25){}="a1"^(.75){}="b1" \ar@{=>} "a1";"b1"\\
% & \mbox{PBW}\ar@{}[dr]^(.25){}="a5"^(.75){}="b5" \ar@{=>} "a5";"b5" &&\mbox{Koszul filtration}\ar@{}[dl]^(.25){}="a2"^(.75){}="b2" \ar@{=>} "a2";"b2"&\\
% && \mbox{Koszul flag}\ar@{}[d]^(.25){}="a"^(.75){}="b" \ar@{=>} "a";"b"&&\\
% && \mbox{Koszul}&&
% }
% \]
% \end{center}
Interestingly, the cohomology algebras in \eqref{eq:cohom superpyth} and \eqref{eq:cohom level2} provide examples of universally Koszul algebras that are not strongly Koszul (see Propositions \ref{prop:superpythagorean filtration} and \ref{prop:level2 filtration} and Propositions \ref{prop:superpythagorean not strongly Koszul} and \ref{prop:level2 not strongly Koszul}).
We are not aware of any previous examples of this behaviour in the literature.

\begin{thmA}
Suppose that $G_F(p)$ is either a free pro-$p$ group or a Demushkin pro-$p$ group. Then $H^\bullet(G_F(p),\F_p)$ is universally Koszul and strongly Koszul.
\end{thmA}
This is proved in Subsections \ref{ssec:uK free}, \ref{ssec:uK Demushkin}, \ref{ssec:sK free}, \ref{ssec:sK Demushkin}.

Free pro-$p$ groups and Demushkin pro-$p$ groups are the building blocks of all known examples of finitely generated groups of type $G_F(p)$. Thus it has been conjectured that all finitely generated maximal pro-$p$ quotients of absolute Galois groups of fields satisfying Hypothesis \ref{ass} are obtained by assembling free and Demushkin pro-$p$ groups via free products and certain semidirect products. This is the Elementary Type Conjecture for pro-$p$ groups (see Section \ref{sec:ET groups} for details). On the level of cohomology, the operation corresponding to a free product of pro-$p$ groups is the direct sum of algebras, and the operation corresponding to the aforementioned semidirect product is the twisted extension of algebras, which often reduces to a skew-symmetric tensor product with an exterior algebra (Definition \ref{def:operations on algebras} and Remark \ref{rem:twisted ext varia}).

\begin{thmB}
The direct sum of any two universally Koszul algebras (respectively, strongly Koszul algebras, algebras with a Koszul filtration) is universally Koszul (respectively, is strongly Koszul, has a Koszul filtration).
\end{thmB}
This is proved in Propositions \ref{prop:direct sum universally Koszul}, \ref{prop:direct sum strongly Koszul} and \ref{prop:direct sum Koszul filtration}.
\begin{thmC}
Any twisted extension of a universally Koszul algebra (respectively, an algebra with a ``good'' Koszul filtration) is universally Koszul (respectively, has a ``good'' Koszul filtration).
\end{thmC}
This is proved in Proposition \ref{prop:twisted extension universally Koszul} and in Proposition \ref{prop:twisted extension Koszul filtration}. By ``good'' Koszul filtration, we mean a Koszul filtration which satisfies the condition labeled by \ref{eq:good filtration} in Proposition \ref{prop:twisted extension Koszul filtration}.

The analogous result for strong Koszulity does not hold in general: the cohomology algebra of a rigid field of level $0$ or $2$ is a twisted extension of the strongly Koszul algebra $\F_2\gen{t}$ or $\F_2\gen{t\mid t^2}$ respectively, yet we prove the following (see Propositions \ref{prop:superpythagorean not strongly Koszul} and \ref{prop:level2 not strongly Koszul}).
\begin{propD}
Let $F$ be a $2$-rigid field of level $0$ or $2$, such that $G_F(2)$ has a minimal set of generators with at least $3$ elements. Then $H^\bullet(G_F(2),\F_2)$ is not strongly Koszul.
\end{propD}
Note that this negative result does not disrupt the general harmony of our picture. In fact, strong Koszulity imposes the too strict condition of working only within a specific set of generators. This restriction was natural in the setting of the paper \cite{hehire} in which strong Koszulity was introduced, but this is not always so in our context.

The aforementioned examples of universally Koszul algebras that are not strongly Koszul arise as the cohomology algebras of the previous proposition.

Moreover, strong Koszulity is preserved by twisted extensions whenever they reduce to skew-symmetric tensor products.
\begin{thmE}
The skew-symmetric tensor product of a strongly Koszul algebra with a finitely generated exterior algebra is strongly Koszul.
\end{thmE}
This is Proposition \ref{prop:twisted extension strongly Koszul}; the analogous statement for universally Koszul algebras, and for algebras with Koszul filtrations, is a corollary of Theorem C.

All together, these results prove that under the Elementary Type Conjecture for pro-$p$ groups, the algebra $H^\bullet(G_F(p),\F_p)$ is universally Koszul for any finitely generated $G_F(p)$, and hence in particular it is Koszul. So, as a bonus, we also obtain a new proof of the main result of \cite{Koszul1}. But the different perspective adopted here adds deep insights into the homological behaviour of modules over Galois cohomology. In fact an algebra is Koszul if its augmentation ideal, or equivalently its ground field, has a remarkably good resolution (a linear resolution, see Definition \ref{def:linear resolution}). As pointed out by Piontkovski\u\i\ in \cite{piontk}, algebras that satisfy enhanced versions of Koszulity have more modules with linear resolutions. The Galois cohomology of a pro-$p$ group as above enjoys this property at its maximum: an algebra is universally Koszul if every ideal generated by elements of degree $1$, or equivalently every cyclic module, has a linear resolution.

There are specific situations in which some form of the Elementary Type Conjecture is known to be true. Thus in any of these situations we get unconditional results. Some significant examples are included in Section \ref{sec:unconditional} and summarised in the next theorem.
\begin{thmF}
If $F$ has characteristic $\chr F\neq 2$ and either
\begin{enumerate}
\item $|F^\times/(F^\times)^2|\leq 32$, or
\item $F$ has at most $4$ quaternion algebras, or
\item $F$ is Pythagorean and formally real, or more generally
\item $E\subseteq F\subseteq E(2)$, with $E$ a Pythagorean and formally real field and $F/E$ a finite extension,
\end{enumerate}
then $H^\bullet(G_F(2),\F_2)$ is universally Koszul.
\end{thmF}

The Bloch-Kato Conjecture and its proof are marvelous advances in Galois theory, however, it seems to us that they are not the completion of the story and that much stronger properties hold for Galois cohomology. We believe that universal Koszulity is the right way to strengthen the Bloch-Kato Conjecture, independently of the Elementary Type Conjecture.
\begin{conj}
In the standing hypothesis that $F$ contains a primitive $p$\textsuperscript{th} root of unity and $G_F(p)$ is finitely generated, then $H^\bullet(G_F(p),\F_p)$ is universally Koszul.
\end{conj}
It even may be possible to extend most of our results to infinitely generated maximal pro-$p$ quotients of absolute Galois groups (in fact, in \cite{positselsky:koszul2}, Positselski has already successfully proved the PBW property of the Galois cohomology of general local and global fields). This possible extension, along with the previous conjecture, is research in progress.

We made a special effort to write this paper in such a way that mathematicians from various areas such as Galois theory, commutative algebra, algebraic geometry, and computer science may find it accessible.
% The purpose of this paper is to investigate some properties related to Koszulity for the algebras $H^\bullet(G_F,\F_p)$, or equivalently $H^\bullet(G_F(p),\F_p)$. These properties are the original Koszulity, the existence of Koszul flags, the existence of Koszul filtrations, universal Koszulity, strong Koszulity, PBW property. They appeared in \cite{priddy}, \cite{conca:u-Koszul}, \cite{hehire}, \cite{corova}, \cite{cotrva}, \cite{piontk}, \cite{codero}, \cite{enhehi}
\subsection*{Acknowledgements}
We are very grateful to Sunil Chebolu, Ronie Peterson Dario, Ido Efrat, David Eisenbud, Jochen G\"artner, Stefan Gille, Pierre Guillot, Chris Hall, Daniel Krashen, John Labute, Eliyahu Matzri, Danny Neftin, Jasmin Omanovic, Claudio Quadrelli, Andrew Schultz, John Tate, \c{S}tefan Toh\v{a}neanu, Adam Topaz, Thomas Weigel and Olivier Wittenberg for interesting discussions, enthusiasm, and encouragement concerning this paper and related topics.
The authors are extremely grateful to the referee for her/his insightful report and for the very careful and thoughtful comments, suggestions and corrections. These suggestions and corrections have helped us to make our paper clearer and more readable. J\'an Min\'a\v{c} gratefully acknowledges support from Natural Sciences and Engineering Research Council of Canada (NSERC) grant R0370A01.
   Nguy$\tilde{\text{\^{e}}}$n Duy T\^{a}n is supported by the Vietnam National Foundation for Science and Technology Development (NAFOSTED) under grant number 101.04-2019.314 and the Vietnam Academy of Science and Technology under grant number CT0000.02/20-21.

%Some of these notions are very restrictive even among Koszul algebras. Surprisingly, in our paper

\section{Koszulity and relatives}\label{sec:koszul}
In this section, $K$ is an arbitrary field and we consider possibly noncommutative algebras over $K$. Ideals and modules are to be assumed as \emph{left} ideals or modules unless otherwise specified. Enhanced Koszul properties are defined accordingly. We warn the reader that some of them are not left-right symmetric: they have a version for right ideals and modules, and algebras satisfying the left version of a property need not satisfy the corresponding right version, and vice versa. For example, in \cite{pio:noncomm} it is asserted that the algebra $K\gen{x,y,z,t\mid zy-tz,zx}$ is universally Koszul for left ideals but not for right ideals. This problem becomes irrelevant in the cohomology of absolute Galois groups, as the algebras $H^\bullet(G_F,\F_p)$ are graded-commutative (\cite{milnor}).
\subsection{Original Koszul property}
An associative, unital $K$-algebra $A$ is \emph{graded} if it decomposes as a direct sum of $K$-vector spaces $A=\bigoplus_{i\in\Z}A_i$ such that, for all $i,j\in\Z$, $A_iA_j\subseteq A_{i+j}$. $A_i$ is the \emph{homogeneous component} of degree $i$. A graded $K$-algebra $A$ is \emph{connected} if $A_i=0$ for all $i<0$ and $A_0=K\cdot 1_A\cong K$; \emph{generated in degree $1$} if all of its (algebra) generators are in $A_1$; \emph{locally finite-dimensional} if $\dim_{K} A_i<\infty$ for all $i\in\Z$ (a common alternative terminology for this is \emph{finite-type}).
A left module $M$ over a graded algebra $A$ is \emph{graded} if it decomposes as a direct sum of $K$-vector spaces $M=\bigoplus_{i\in\Z}M_i$ such that, for all $i,j\in\Z$, $A_iM_j\subseteq M_{i+j}$. A graded left $A$-module $M$ is \emph{connected} if there is a $n\in\Z$ such that $M_i=0$ for all $i<n$; \emph{generated in degree $n$} if all its (module) generators are in $M_n$; \emph{locally finite-dimensional} or \emph{finite type} if $\dim_{K} M_i<\infty$ for all $i\in\Z$.

For a vector space $V$ over $K$, let $T(V)$ denote the \emph{tensor} $K$-algebra generated by $V$, that is,
\[T(V)=\bigoplus_{n\geq0}V^{\otimes n},\qquad \text{with }V^{\otimes0}=K.\]
This is a graded algebra with algebra multiplication given by the tensor product. Whenever possible, in writing the elements of a tensor algebra, we will omit tensor signs, and simply use juxtaposition. In particular, if $V$ is finite-dimensional with a basis $\{x_1,\dots,x_n\}$, then $T(V)$ will be identified with the algebra $K\gen{x_1,\dots,x_n}$ of noncommutative polynomials in the variables $x_1,\dots,x_n$, graded by polynomial degree.

Ignoring the grading, tensor algebras are the free objects in the category of (associative, unital) algebras (see \cite[\S 1.1.3]{lodval}), so any algebra is a quotient of a tensor algebra. When an algebra is presented as a quotient of a tensor algebra, we will simply identify an element of the former with each of its representatives in the latter. If the aforementioned quotient is over a two-sided ideal generated by homogeneous elements, then the new algebra inherits a well-defined grading from the tensor algebra.

For any graded algebra $A$, there exists a unique morphism of graded algebras $\pi:T(A_1)\to A$ that is the identity in degree $1$. Clearly, $A$ being generated in degree $1$ is equivalent to $\pi$ being surjective. The algebra $A$ is \emph{quadratic} if it is generated in degree $1$ and $\ker\ \pi=(\ker\ \pi_2)=T(A_1)\ker\ \pi_2T(A_1)$. In this case, we call $V=A_1$ the \emph{space of generators} and $R=\ker\ \pi_2$ the \emph{space of relators}, and we use the notation $A=Q(V,R)$.

Henceforth we assume graded algebras to be locally finite-dimensional, connected and generated in degree $1$, unless otherwise specified. In particular, every such algebra $A$ is equipped with the \emph{augmentation map} $\varepsilon\colon A\to K$, that is the projection onto $A_0$. Its kernel is the \emph{augmentation ideal} $A_+=\bigoplus_{n\geq1}A_n$.
The augmentation map gives $K$ a canonical structure of graded $A$-bimodule concentrated in degree $0$, via the actions $a\cdot k=\varepsilon(a)k$ and $k\cdot a=k\varepsilon(a)$ $(a\in A, k\in K)$. In what follows, ground fields of connected graded algebras will always be understood to be equipped with this graded bimodule structure, or just the left module part thereof.

A graded left module $M$ over a quadratic algebra $A$ is generated in degree $n$ if and only if the natural map $\varpi:A\otimes M_n\to M$ induced by the module structure is surjective. This of course forces $M_i=0$ for all $i<n$. $M$ is \emph{quadratic} if, for some $n\in\Z$, $M$ is generated in degree $n$ and $\ker\ \varpi$ is generated, as an $A$-submodule, by $\ker\ \varpi\cap (A_1\otimes M_n)$. Informally, this means that $M$ has a presentation with homogeneous degree $n$ generators and homogeneous degree $n+1$ defining relators.

We adhere to the convention that differentials of complexes of graded modules have degree $0$.

\begin{defn}\label{def:linear resolution}
Let $A$ be a graded $K$-algebra and $M$ a graded left $A$-module generated in degree $n$. A resolution $(P_\bullet, d_\bullet)$ of $M$ by free graded left $A$-modules is said to be \emph{linear} if each $P_i$ is generated in degree $i+n$: $P_i=A\cdot(P_i)_{i+n}$.
\end{defn}
\begin{rem}
This terminology is due to the fact that, if the modules $M_i$ are finitely generated, the entries of the matrix representing each $d_i$ in suitable bases are all in $A_1$.
\end{rem}
\begin{defn}\label{def:koszul algebra}
A quadratic $K$-algebra $A$ is \emph{Koszul} if $K$, as a trivial graded left $A$-module concentrated in degree $0$, has a linear resolution.
\end{defn}
\begin{defn}\label{def:koszul module}
A graded left module $M$ generated in degree $n$ over a quadratic algebra $A$ is \emph{Koszul} if it has a linear resolution.
\end{defn}
\begin{rem}
Our definition differs from the original definition in \cite[Section 2.1]{polpos} in that, following \cite{piontk}, we allow modules generated in arbitrary degree to be Koszul. However, this difference is not serious. In fact, for any graded left module $M$ and any integer $z$, we can define the \emph{shifted module} $M(z)$, that is the same module with shifted grading: $M(z)_i=M_{z+i}$. A module is Koszul according to our definition, if and only if its appropriate shifted module is Koszul according to \cite{polpos}.
%
%Note also that the existence of a linear resolution is equivalent\marginpar{explain better or cite something!} to the existence of a linear \emph{free} resolution (see \cite[Section 2.1]{polpos}).
\end{rem}
A free resolution of $K$ can be contracted to a free resolution of $A_+$ and, conversely, a free resolution of $A_+$ can be extended to a free resolution of $K$. These operations on resolutions preserve the property of being linear (see also the proof of Proposition \ref{prop:equiv to uK}). Hence the following statements are equivalent:
\begin{enumerate}
\item $A$ is Koszul;
\item $K$ is a Koszul left module;
\item $A_+$ is a Koszul left module.
\end{enumerate}

\begin{rem}\label{rem:koszulity and tor}
Since in our assumptions the ground field of an algebra, as a graded left module, is concentrated in degree $0$, Definition \ref{def:koszul module} is an extension of Definition \ref{def:koszul algebra}. Moreover, the existence of linear resolutions can be equivalently expressed in homological terms. Since differentials are assumed to be of degree $0$, it is possible to slice a complex of graded modules into its subcomplexes of vector spaces given by the homogeneous components. As a consequence, the Tor functors of a graded module $M$ can be bigraded as follows. Given a graded right module $R$ and a resolution $(P_\bullet, d_\bullet)$ of the graded left module $M$ by free graded left modules, the tensor product complex $(R\otimes_A P_\bullet,\mathrm{id}_N\otimes d_\bullet)$ is bigraded: besides the homological degree, there is the \emph{internal} degree 
\[
(R\otimes_A P_i)_j=\oplus_{h+k=j}N_h\otimes_A (P_i)_k.
\]
The differential $\mathrm{id}_N\otimes d_i$ maps the component of $R\otimes_A P_\bullet$ of bidegree $(i,j)$ to the component of bidegree $(i-1,j)$, so it is possible to define
\[
\mathrm{Tor}^A_{i,j}(R,M)=\frac{\ker\ d_i:(R\otimes_A P_i)_j\to(R\otimes_A P_{i-1})_j}{\im\ d_{i+1}:(R\otimes_A P_{i+1})_j\to(R\otimes_A P_i)_j}.
\]
This applies in particular to the right module $R=K$ with action given by the augmentation map. Since $K$ is concentrated in degree $0$, we obtain
\[
\mathrm{Tor}^A_{i,j}(K,M)=\frac{\ker\ d_i:(K\otimes_A P_i)_j\to(K\otimes_A P_{i-1})_j}{\im\ d_{i+1}:(K\otimes_A P_{i+1})_j\to(K\otimes_A P_i)_j}.
\]
Now, suppose that $M$ is generated in degree $n$ and $(P_\bullet, d_\bullet)$ is a linear resolution, so that each $P_i$ is generated in degree $i+n$. Let $x\in(P_i)_j$ with $j>i+n$. Then there are $a_t\in A_+$ and $z_t\in (P_i)_{i+n}$ such that $x=\sum_ta_tz_t$. But then, for all $k\in K$, $k\otimes x=k\otimes \sum_ta_tz_t=\sum_t(k\varepsilon(a_t)\otimes z_t)=0$. In other words, $(K\otimes_A P_i)_j=0$, and hence $\Tor_{i,j}^A(K,M)=0$. Viceversa, if $M$ is generated in degree $n$ and $\Tor_{i,j}^A(K,M)=0$ for all $i\geq0$ and $j\neq i+n$, then the standard minimal resolution of $M$ constructed inductively via free covers of kernels (see \cite[Section 1.4]{polpos}) is a linear resolution. To sum up, a graded left module $M$ generated in degree $n$ over a quadratic $K$-algebra $A$ is Koszul, if and only if $H_i(M)_j:=\Tor_{i,j}^A(K,M)=0$ for all $i\geq0$ and $j\neq i+n$. From this it follows that a Koszul module is necessarily quadratic.
\end{rem}

\subsection{Enhanced forms of Koszulity}
The Koszul property is difficult to check by the definition. It turns out that there are properties that imply Koszulity and are usually easier to check. The first one that we consider is a form of the Poincar\'e-Birkhoff-Witt (PBW) Theorem.

In the theory of Lie algebras, the PBW Theorem asserts that, if $\{x_i\mid i\in I\subseteq\N\}$ is a vector space basis of a Lie algebra $L$, totally ordered by, for example, $x_i<x_j$ for $i<j$, then the set of monomials 
\[
\{x_{i_1}^{k_1}\dots x_{i_l}^{k_l}\mid x_{i_1}<\dots<x_{i_l}, k_i\in\N\setminus\{0\}\}\cup\{1\}
\]
is a basis of the universal enveloping algebra $U(L)$.

There is a class of quadratic algebras, introduced in \cite{priddy}, that exhibit an analogous PBW-type vector space basis. Here we give a brief outline of their behaviour, referring to  \cite[Chapter 4]{lodval}, \cite[Chapter 4]{polpos} or \cite[Section 2.4]{Koszul1} for more details.
Let $A=Q(V,R)$ be a quadratic algebra. We choose a basis $\mathcal B$ of the space of generators $V$ and a total order on $\mathcal B$. We extend the total order to an admissible order on all monomials of $T(V)$. A total order $\preceq$ on the set $\mathrm{Mon}(T(V))$ of monomials of $T(V)$ is said to be \emph{admissible} if
\begin{enumerate}
\item the restriction of $\preceq$ to $\mathcal B$ coincides with the given total order on that basis;
\item for all $a\in \mathrm{Mon}(T(V))$, $1\preceq a$;
\item for all $a,b,c\in \mathrm{Mon}(T(V))$, \(
a\preceq b\Rightarrow ac\preceq bc, ca\preceq cb.
\)
\end{enumerate}
The $\preceq$-maximal monomial in each element $r\in R$ is called the \emph{leading monomial} of $r$. By the Gram-Schmidt process, we can always work with a basis of $R$ such that the leading monomial of each basis vector has coefficient $1$ and does not appear in any other basis vector. Such a basis is referred to as a \emph{normalized basis}. Once we have a normalized basis $\mathcal D$, each defining relation $r=0$, $r\in\mathcal D$, of $A$ can then be interpreted as a \emph{reduction rule} that substitutes its leading monomial with a linear combination of $\preceq$-lower monomials. A monomial in $T(V)$ is said to be \emph{reduced} if it is not possible to apply any reduction rule to it, or equivalently if it does not contain any leading monomial as a submonomial. Any monomial in $T(V)$ can be turned into a linear combination of reduced monomials via the application of finitely many reduction rules, thanks to the axioms of monomial order and the local finite dimensionality of $A$. Hence reduced monomials always linearly span $A$. The interesting question is when reduced monomials are linearly independent.
\begin{defn}\label{def:PBW basis}
Keeping the previous notation, the ordered basis $\mathcal B$ of $V$ is called a set of \emph{PBW generators} of $A$ if, for some admissible order, the reduced monomials of $T(V)$ form a basis of $A$, which is then called a \emph{PBW basis} of $A$.
\end{defn}
Note that the existence of a PBW basis depends on the choice of $\mathcal B$, the choice of a total order on it, and the choice of an admissible extension of such a total order. It may very well happen that some combinations of choices do not produce a PBW basis, while other combinations do.
\begin{defn}\label{def:PBW algebra}
A quadratic algebra is \emph{PBW} if it admits a set of PBW generators.
\end{defn}

The advantage of the PBW property lies in the fact that the linear independence of reduced monomials is completely determined by what takes place in degree $3$.
\begin{thm}[Bergman's Diamond Lemma, \cite{bergman}]
Let $A=Q(V,R)$ be a quadratic algebra. Choose a totally ordered basis $\mathcal B$ of $V$ and an admissible extension of the total order to $\mathrm{Mon}(T(V))$. If the set of reduced monomials of degree $3$ is linearly independent, then the set of all reduced monomials is linearly independent. In particular, $A$ is PBW.
\end{thm}
This paves the way for a finite algorithmic procedure that decides whether a given ordered basis of $V$ is a set of PBW generators of $A$ with respect to a given admissible order extension to $\mathrm{Mon}(T(V))$. This algorithm is the \emph{Rewriting Method} of \cite[Section 4.1]{lodval}. In fact, the reduced monomials of degree $3$ are linearly dependent if and only if there exists a non-reduced monomial of degree $3$ that can be reduced in two distinct ways not terminating with the same linear combination of reduced monomials. A monomial of degree $3$ is reducible according to two distinct sequences of reduction rules exactly when both its first $2$ letters and its last $2$ letters are leading monomials: in this case the first sequence starts reducing the first $2$ letters and the second sequence starts reducing the last $2$ letters. Such monomials are called \emph{critical}. Each critical monomial $m$ gives rise to an oriented graph, whose vertices are $m$ itself and the polynomials obtained from $m$ after each step of each reduction sequence, and whose edges are the reduction rules applied. The terminal vertices of the graph are precisely the linear combinations of reduced monomials that are equivalent to $m$ in $A$. The graph is \emph{confluent} if it has a unique terminal vertex. So the chosen ordered basis $\mathcal B$ is a set of PBW generators of $A$ for the chosen admissible order extension if and only if the graph of each critical monomial is confluent.
The original proof that a quadratic PBW algebra is Koszul is due to Priddy \cite[Section 5]{priddy}.

For the interested reader, it is worth mentioning that, for quadratic algebras, the concept of PBW basis is dual to the powerful notion of noncommutative Gr\"obner basis (\cite{bergman}, \cite{bucwin}). A \emph{(noncommutative) Gr\"obner basis} of a two-sided ideal $I$ of a tensor algebra is a set $G\subseteq I$ that generates $I$, and such that the leading monomials of the elements of $G$ and the leading monomials of the elements of $I$ generate the same two-sided ideal.
A quadratic algebra $A=Q(V,R)$ is \emph{PBW} if and only if the two-sided ideal $(R)$ has a quadratic (noncommutative) Gr\"obner basis (see \cite[Section 4.3.7]{lodval}).

Changing perspective, there is a family of enhanced Koszul properties based on ``divide and conquer'' strategies. We first recall a concept these properties rely on.
\begin{defn}\label{def:colon ideal}
Let $A$ be a $K$-algebra, $I\id A$ a left ideal, and $x\in A$. The \emph{colon ideal} $I:x$ is $\{a\in A\mid ax\subseteq I\}$. If $A$ is graded-commutative and $J\id A$ is another ideal, there is a more general concept of the \emph{colon ideal}: $I:J=\{a\in A\mid aJ\subseteq I\}$.
\end{defn}
\begin{rem}\label{rem:homogeneous colon ideals}
If $I\id A$ and $x\in A$ are homogeneous, then the colon ideal $I:x$ is also homogeneous. Analogously, in the graded-commutative setting, if $I\id A$ and $J\id A$ are homogeneous, then the colon ideal $I:J$ is also homogeneous.
\end{rem}

\begin{defn}\label{def:linear module}
Let $A$ be a quadratic algebra and let $X=\{u_1,\dots,u_d\}$ be a minimal system of homogeneous generators of the augmentation ideal $A_+$.
A graded left $A$-module $M$ is \emph{linear} if it admits a system of homogeneous generators $\{g_1,\dots,g_m\}$, all of the same degree, such that, for each $\{i_1,\dots,i_k\} \subseteq \{1,\dots,m\}$, the left ideal of $A$
\[
(Ag_{i_1}+\dots+Ag_{i_{k-1}})\colon g_{i_k}=\{a\in A\mid ag_{i_k}\in Ag_{i_1}+\dots+Ag_{i_{k-1}}\}
\]
is generated by a subset of $X$. Such a system is called a \emph{set of linear generators} of $M$.
\end{defn}
\begin{defn}\label{def:strongly koszul}
A quadratic algebra $A$ is called \emph{strongly Koszul} if its augmentation ideal $A_+$ admits a minimal system of homogeneous generators $X=\{u_1,\dots,u_d\}$ such that for every subset $Y\subsetneq X$, and for every $x\in X\setminus Y$, the colon ideal $(Y)\colon x$ is generated by a subset of $X$.
% A quadratic $\F$-algebra $R=S/I$ is called \emph{strongly Koszul} if its augmentation ideal $R_+$ admits a minimal system of homogeneous generators $X=\{u_1,\dots,u_n\}$ such that for every $Y\subsetneq X$ and for every $u_i\in X\setminus Y$ the quotient ideal
% \[
% (Y)\colon u_i=\{a\in R\mid au_i\in(Y)\}
% \]
% is generated by a subset of $X$.
\end{defn}
Our definition of strongly Koszul is taken from \cite{codero}, and it differs from the original definition of \cite{hehire}. In \cite{hehire}, the system of generators is assumed to be totally ordered, and the property that $(u_{i_1},\dots,u_{i_{r-1}})\colon u_{i_r}$ is generated by a subset of $X$ is only required for $u_{i_1}<\dots<u_{i_{r-1}}<u_{i_r}$. Since in Galois cohomology there is no natural total order on generators, we prefer the given version.
\begin{rem}\label{rem:strongly koszul and linear modules}
Keeping the same notation as in the previous definitions, $A$ is strongly Koszul if and only if all left ideals $(u_{i_1},\dots,u_{i_{r-1}})$ are linear left $A$-modules.
\end{rem}
In \cite{hehire} it is also proved that, over a strongly Koszul commutative $K$-algebra with a minimal system of homogeneous generators $\{u_1,\dots,u_d\}$, any ideal of the form $(u_{i_1},\dots,u_{i_{r}})$ has a linear resolution, and in particular $A$ is Koszul.
In the general (that is, possibly noncommutative) setting, it is immediate that strong Koszulity implies the existence of a \emph{Koszul filtration}, defined below.
\begin{defn}\label{def:koszul filtration}
A collection $\mathcal{F}$ of left ideals of a quadratic algebra $A$ is a \emph{Koszul filtration} if
\begin{enumerate}
\item each ideal $I\in\mathcal F$ is generated by elements of $A_1$;
\item the zero ideal and the augmentation ideal $A_+$ belong to $\mathcal F$;
\item for each nonzero ideal $I\in \mathcal F$, there exist an ideal $J \in \mathcal F$, $J\neq I$, and an element $x\in A_1$ such that $I = J + Ax$, and the ideal
\[
J\colon x=\{a \in A\mid ax \in J\}
\]
lies in $\mathcal F$.
\end{enumerate}
\end{defn}
The concept of Koszul filtration was first introduced for commutative algebras in \cite{cotrva} as a more flexible version of strong Koszulity. In fact the definition precisely encodes the properties of strong Koszulity needed to imply the Koszulity of the relevant ideals, but the formulation is coordinate-free, allowing more choice in the collection of ideals. The definition in the general setting appeared in \cite{piontk}, as well as the proof that each ideal of a Koszul filtration is a Koszul module (there the author works with right ideals, but of course everything carries over to left ideals).

In particular, out of a Koszul filtration, we can choose ideals with any number of generators in a range from $1$ to the minimal number of generators of the algebra. This leads to the following definition, which also appeared in \cite{piontk}.
\begin{defn}\label{def:koszul flag}
Let $A$ be a quadratic algebra with a $d$-dimensional space of generators.
A sequence of left ideals $(I_0,\dots,I_d)$ of $A$ is a \emph{Koszul flag} if 
\begin{enumerate}
\item $(0)=I_0<I_1<\dots<I_d=A_+$;
\item for all $k$, there is $x_k\in A_1$, such that $I_{k+1}=I_k+(x_{k+1})$;
\item for all $k$, $I_k$ has a linear free graded resolution.
\end{enumerate}
\end{defn}

Algebras satisfying one of the preceding conditions have sufficiently many Koszul cyclic left modules. At the extreme, one may ask for the algebra to have a Koszul filtration as large as possible, so that all left ideals generated in degree $1$, or equivalently all cyclic left modules, are Koszul. In the commutative setting, this property has been introduced in \cite{conca:u-Koszul} under the name of \emph{universal Koszulity}. We extend this definition to the general setting.
\begin{defn}
A quadratic algebra $A$ is called \emph{universally Koszul} if every left ideal generated in degree $1$ has a linear resolution.
\end{defn}
%Although there are algebras that are left but not right universally Koszul, or vice versa (\cite[Section 7]{piontk}), the theory of left universal Koszulity
As in the commutative setting in \cite{conca:u-Koszul}, we define
\[
\mathcal L(A)=\{I\id A\mid I=AI_1\}
\]
to be the set of all left ideals of $A$ generated in degree $1$, and we have the following characterisation of \uK algebras, which is proved exactly as in the commutative setting (\cite[Proposition 1.4]{conca:u-Koszul}).
\begin{prop}\label{prop:equiv to uK}
Let $A$ be a quadratic $K$-algebra. The following conditions are equivalent:
\begin{enumerate}
\item $A$ is \uK;
\item\label{cond2} for every $I\in\mathcal L(A)$, one has $\Tor_{2,j}^A (A/I, K) = 0$ for every $j > 2$ (here $A/I$ is declared to be generated in degree $0$);
\item\label{cond3} for every $J\in\mathcal L(A)$ and every $x\in A_1\setminus J$, one has $J:x\in\mathcal{L}(A)$;
\item\label{cond4} $\mathcal L(A)$ is a Koszul filtration.
\end{enumerate}
\end{prop}
\begin{proof}\
\begin{enumerate}
\item[(1)$\Rightarrow$(2)] By general principles of homological algebra, a linear resolution of $I$, which exists by definition of universal Koszulity, can be extended to a linear resolution of $A/I$:
\[
\xymatrix@R-8pt@C+8pt{
\dots\ar[r] & P_2\ar^{d_2}[r] & P_1\ar^{d_1}[r]\ar@{=}[ldd] & P_0\ar^{d_0}@{>>}[r]\ar@{=}[ldd] & I\ar@{^{(}->}^{\iota}[d]\ar@{_{(}->}^{\iota}[ldd] \\
& & & & A\ar@{>>}^{\pi}[d]\\
\dots\ar[r] & P_1\ar^{d_1}[r]\ar@{>>}[dr] & P_0\ar^{\alpha=\iota d_0}[r]\ar@{>>}_{d_0}[dr] & A\ar^{\pi}@{>>}[r] & A/I.\\
& \dots & \ker\ d_0\ar@{^{(}->}[u] & I\ar@{^{(}->}_{\iota}[u]& 
}
\]
Here $A$ is generated in degree $0$ as a left $A$-module, so that $\alpha$ maps the elements of degree $1$ (the generators of $P_0$) to elements of degree $1$ (the generators of $I$ inside $A$). In other words, the map $\alpha$ respects our convention that differentials have degree $0$.
Note that $I$ is generated in degree $1$, while $R/I$ is generated in degree $0$, so there is no need to artificially alter the degrees of the $P_i$'s in the resolution of $R/I$ (compare this with Definition \ref{def:linear resolution}). In particular, each $P_i$ is generated in degree $i+1$ in both resolutions. Now (2) is a particular case of Remark \ref{rem:koszulity and tor}.
\item[(2)$\Rightarrow$(3)] Let $J\in\mathcal L(A)$ be generated by $y_1,\dots,y_d\in A_1$ (the system of generators can be chosen to be finite because, in our assumptions, $\dim_K(A_1)<\infty$) and let $x\in A_1\setminus J$. A presentation $\phi:A\oplus A^d\twoheadrightarrow Ax+J$ defines the first step of a (minimal) free resolution
\[
\xymatrix{
Ax\oplus (Ay_1\oplus\dots\oplus Ay_d)\ar^-{d_0}[r] & Ax+J.
}
\]
The free left $A$-module on a set of generators of $\ker\ d_0$ serves as the graded module in homological degree $2$ in a free resolution of $A/(Ax+J)$. Therefore, by (2), $\ker\ d_0$ has to be generated by elements of degree $2$ in $A$. Correspondingly, the syzygy module $\ker\ \phi$ is generated by sums of elements of degree $1$ in $A$. But $J:x$ is the projection of $\ker\ \phi$ onto the first coordinate, so it has to be generated by elements of degree $1$ in $A$.
\item[(3)$\Rightarrow$(4)] The first two conditions in the definition of Koszul filtration are trivially satisfied. As regards the third, let $I=Ax_1+\dots+Ax_n\in\mathcal{L}(A)\setminus\{(0)\}$, and suppose that no proper subsets of $\{x_1,\dots,x_n\}\subseteq A_1$ generate $I$. Setting $J=Ax_2+\dots+Ax_n$, we have $I=J+Ax_1\neq J$, and $J:x\in\mathcal L(A)$ by (\ref{cond3}).
\item[(4)$\Rightarrow$(1)] This is Proposition 2.2 in \cite{piontk}.
\end{enumerate}

\end{proof}

% \begin{por}\label{por:various}
% Let $A$ be a quadratic algebra and $I\id A$ an ideal. 
% \begin{enumerate}
% \item If $I$ has a linear resolution, then $\Tor_{2,j}^A (A/I, \F) = 0$.
% \item If $\Tor_{2,j}^A (A/I, \F) = 0$ and $I=(a_1,\dots,a_d)\neq (0)$, then for all $i=1,\dots,d$ the colon ideal $(a_1,\dots,a_{i-1},a_{i+1},\dots,a_d):a_i$ is generated in degree $1$.
% \item Let $\mathcal M\subseteq\mathcal L(A)$ be a subfamily of ideals. If for every $I\in\mathcal M$ and every $x\in A_1\setminus I$ one has $I:(x)\in\mathcal{L}(A)$
% \end{enumerate}
% \end{por}

In our treatment of enhanced Koszulity, we will often deal with extensions of algebras, and in particular with some ideals of the smaller algebra that have to be extended to the larger one, or vice versa. These two processes are defined below. 
\begin{defn}
Let $A\subseteq B$ be two $K$-algebras. Given a left ideal $I$ of $A$, we define its \emph{extension} to $B$ to be the intersection $I^e$ of all left ideals of $B$ that contain $I$. In other words, $I^e$ is the ideal of $B$ generated by $I$. Given a left ideal $J$ of $B$, we define its \emph{contraction} to $A$ to be the left ideal $J^c=J\cap A$ of $A$. In case $A_1, A_2$ are two subalgebras of $B$, we will use the notation $J^c_{A_1}$ or $J^c_{A_2}$ to differentiate between the contractions of $J$ to either subalgebra.
\end{defn}
In the same spirit, when dealing with colon ideals, when we need to specify the algebra in which they are defined, we add the name of that algebra as a subscript to the colon sign.

\section{Elementary type pro-$p$ groups}\label{sec:ET groups}
Recall that $G_F(p)$ is the maximal pro-$p$ quotient of the absolute Galois group of $F$.
The class of all finitely generated pro-$p$ groups of the form $G_F(p)$, for some field $F$ as in Hypothesis \ref{ass}, is far from being well understood. A conjectural description of such groups has been proposed in \cite{jacwar} for $p=2$, and in \cite{ido:ETC} for $p$ odd. It is called the \emph{Elementary Type Conjecture}, and it was motivated by the homonymous conjecture for Witt rings of quadratic forms (see Conjecture \ref{conj:ETC Witt}).

%\subsection{Background}
At present, all finitely generated pro-$p$ groups that are known to be realisable as $G_F(p)$, for some field $F$ as in Hypothesis \ref{ass} (in short: \emph{realisable}), belong to a restricted class defined inductively from two basic types of groups: free pro-$p$ groups (\cite[Section I.1.5]{serre}) and Demushkin groups (defined below). All free pro-$p$ groups are realisable (\cite[Section 4.8]{lubvdd}), most over local fields not containing primitive $p$\textsuperscript{th} roots of unity (\cite[Section II.5, Theorem 3]{serre}). On the other hand, the maximal pro-$p$ quotients of absolute Galois groups of local fields containing primitive $p$\textsuperscript{th} roots of unity are Demushkin (\cite[Section II.5, Theorem 4]{serre}), but it is not known whether all Demushkin groups are realisable.

The Elementary Type Conjecture claims that all realisable finitely generated pro-$p$ groups can be built from the previous basic groups, endowed with a character that records the group action on roots of unity. We first introduce 
\begin{defn}\label{defn:cyclochar}
% We denote the multiplicative group of units of $\Z_p$ by $U_p$.
%We say that $\chi$ is a \emph{cyclotomic character} if the map \eqref{eq:group hom} is surjective for any $j\geq 1$.
A \emph{cyclotomic pair} is a couple $(G,\chi)$ made up of a pro-$p$ group $G$ and a continuous homomorphism $\chi$ from $G$ to the multiplicative group of units of $\Z_p$. We call $\chi$ a \emph{cyclotomic character}. %(in \cite{cmq:fast} and \cite{cq:bk} $\chi$ is called an \emph{orientation}).
The vocabulary of group properties is extended to cyclotomic pairs, so that, for example, a finitely generated cyclotomic pair is a cyclotomic pair $(G,\chi)$, such that $G$ is finitely generated as a pro-$p$ group.
\end{defn}
This terminology is justified in view of the situation when $G=G_F(p)$ for a field $F$ as in Hypothesis \ref{ass}. The action of $G_{F}(p)$ on the group $\mu_{p^\infty}$ of the roots of unity of order a power of $p$ lying in $F_{\mathrm{sep}}$ induces a homomorphism $\chi_p:G\to \Aut(\mu_{p^\infty})$. This homomorphism is known in the literature as the ($p$-adic) cyclotomic character.

A finitely generated pro-$p$ group is \emph{free} on a set $I=\{x_1,\dots,x_d\}$ if it is the inverse limit of the groups $F(I)/M$, where $F(I)$ is the discrete free group on $I$ and $M$ runs over the normal subgroups of $F(I)$ with index a (finite) power of $p$. Equivalently, free pro-$p$ groups are the free objects in the category of pro-$p$ groups (see \cite[I.1.5]{serre}, where not necessarily finitely generated free pro-$p$ groups are introduced).

A \emph{Demushkin group} is a finitely generated pro-$p$ group $G$ such that $H^2(G,\F_p)$ is a $1$-dimensional $\F_p$-vector space and the cup product induces a non-degenerate pairing
\[
H^1(G,\F_p)\times H^1(G,\F_p) \to H^2(G,\F_p).
\]
In this situation, we identify $H^2(G,\F_p)$ with $\F_p$.

The classification of Demushkin groups was achieved by J.~Labute (\cite{labute:demushkin}), building on previous contributions by S.~P.~Demushkin (\cite{demushkin1961} and \cite{demushkin1963}) and J.-P.~Serre (\cite{serre:demushkin}).
The only finite Demushkin group is $C_2$, the group of order $2$, which is also the only Demushkin group that can be generated by just $1$ element.(\cite[Proposition 3.9.10]{NSW}). It has a unique cyclotomic character with image $\{-1, +1\}$. Infinite Demushkin groups are classified by two invariants. The first is the minimal number of generators. As regards the second, each Demushkin group $G$ has a unique distinguished cyclotomic character $\chi_G$. It is characterised by the fact that the topological $G$-module $\Z_p$, where the $G$-action is induced by $\chi_G$, has the property that, for all $i\in\N\setminus\{0\}$, the canonical homomorphism $H^1(G,\Z_p/p^i\Z_p)\to H^1(G,\Z_p/p\Z_p)$ is surjective (see \cite[Proposition 6 and Theorem 4]{labute:demushkin} and also \cite{serre:demushkin}). We then define the second
invariant $q = q(G)$ to be $2$ if $\mathrm{Im}\chi_G = \{-1, +1\}$, and to be the maximal power $p^k$ such that $\mathrm{Im}\chi_G\subseteq 1 + p^k \Z_p$ otherwise.
\begin{rem}\label{rem:cohomology of Demushkin}
The classification of Demushkin groups also provides an explicit description of their cohomology. Note that, since $C_2$ is finite, its Galois cohomology with $\F_2$ coefficients coincides with its classical (discrete) cohomology with $\F_2$ coefficients, which is the polynomial ring in one variable over $\F_2$ (\cite[Section III.1, Example 2]{brown}). On the other hand, infinite Demushkin groups have their Galois cohomology with $\F_p$ coefficients concentrated in degrees $0,1,2$. So, by their definition, their cohomology with $\F_p$ coefficients is completely determined by the effect of the cup product on a basis of the first cohomology group.
Let $d$ denote the cardinality of a minimal set of generators of an infinite Demushkin group $G$. There are $3$ cases in which $G$ can fall (see \cite[Proposition 4]{labute:demushkin}).
\begin{enumerate}
\item Suppose first that $G$ is a Demushkin group with second invariant $q\neq2$.
In this case, the minimal number of generators is necessarily even: $d=2k$ with $k\in\N\setminus\{0\}$. Moreover, there exists a symplectic basis $X=\{a_1,\dots,a_{2k}\}$ of $H^1(G,\F_p)$ such that
\[
\begin{array}{l}
a_1\cup a_2=a_3\cup a_4=\dots=a_{2k-1}\cup a_{2k}=1,\\
a_2\cup a_1=a_4\cup a_3=\dots=a_{2k}\cup a_{2k-1}=-1,\\
a_i\cup a_j=0 \mbox{ for all other }i,j.
\end{array}
\]
These Demushkin groups are those studied by Demushkin in \cite{demushkin1961} and \cite{demushkin1963}.
\item Now suppose that $G$ is a Demushkin pro-$2$ group on $d=2k+1$ generators, $k\in\N\setminus\{0\}$, with second invariant $q=2$.
Then there is a basis $X=\{a_1,\dots,a_{2k+1}\}$ of $H^1(G,\F_2)$ such that
\[
\begin{array}{l}
a_1\cup a_1=a_2\cup a_3=a_4\cup a_5=\dots=a_{2k}\cup a_{2k+1}=1,\\
a_3\cup a_2=a_5\cup a_4=\dots=a_{2k+1}\cup a_{2k}=1,\\
a_i\cup a_j=0 \mbox{ for all other }i,j.
\end{array}
\]
These Demushkin groups are those studied by Serre in \cite{serre:demushkin}.
\item Suppose finally that $G$ is a Demushkin pro-$2$ group on $d=2k$ generators, $k\in\N\setminus\{0\}$, with second invariant $q=2$.
Then there is a basis $X=\{a_1,\dots,a_{2k}\}$ of $H^1(G,\F_2)$ such that
\[
\begin{array}{l}
a_1\cup a_1=a_1\cup a_2=a_3\cup a_4=\dots=a_{2k-1}\cup a_{2k}=1,\\
a_2\cup a_1=a_4\cup a_3=\dots=a_{2k}\cup a_{2k-1}=1,\\
a_i\cup a_j=0 \mbox{ for all other }i,j.
\end{array}
\]
These Demushkin groups are those studied by Labute in \cite{labute:demushkin}.
\end{enumerate}
\end{rem}

The \emph{free product} of two cyclotomic pairs $(G_1,\chi_1)$ and $(G_2, \chi_2)$ is the cyclotomic pair $(G_1\ast_pG_2, \chi_1\ast_p\chi_2)$, where $G_1\ast_pG_2$ is the coproduct of $G_1$ and $G_2$ in the category of pro-$p$ groups, and $\chi_1\ast_p\chi_2$ is the character induced by $\chi_1$ and $\chi_2$ via the universal property of the coproduct.

The \emph{cyclotomic semidirect product} of the cyclotomic pair $(G,\chi)$ with $\Z_p^m$, $m\in\N\setminus\{0\}$, is the cyclotomic pair $(\Z_p^m\rtimes G, \chi\circ\pi)$, where $\Z_p^m\rtimes G$ is defined by the rule $g(x_1,\dots,x_m)g^{-1}=(\chi(g)x_1,\dots,\chi(g)x_m)$ for all $g \in G$ and $(x_1,\dots,x_m)\in \Z_p^m$, and $\pi:\Z_p^m\rtimes G\to G$ is the canonical projection.

 %Infinite Demushkin groups
%are precisely the Poincaré duality groups of dimension 2 (cf. [Ser13, I.§ 4.5]).
We are now ready to state the \textbf{Elementary Type Conjecture for pro-$p$ groups}.
At first, we shall formulate it for finitely generated pro-$p$ groups of the form $G_F(p)$ when either $p$ is odd or $p=2$ and $\sqrt{-1}\in F$. % Observe that in this case the $q$ invariant of $G_F(p)$ is not 2.
The reason for the additional hypothesis $\sqrt{-1}\in F$ in the case $p=2$ lies in the different behaviour of the group $\mu_{2^\infty}$ of the roots of unity of order a power of $2$ in $F_{\mathrm{sep}}$, compared to the groups  $\mu_{p^\infty}$ for $p$ odd. In fact, we can identify $\Aut(\mu_{2^\infty})\cong\Z_2^*=\{\pm1\}\times\gen 5 \cong \{\pm1\}\times \Z_2$. If $p=2$ and $\sqrt{-1}\in F$, then the cyclotomic character $\chi_2:G_F(2)\to \Aut(\mu_{2^\infty})$ has image 
\begin{equation}\label{eq:good cyclochar}
\im\ \chi_2\subseteq\{+1\}\times\gen 5\cong\gen 5\cong\Z_2,
\end{equation}
similarly to what happens for $p$ odd. If instead $p=2$ and $\sqrt{-1}\notin F$, then the values of $\chi_2$ would be in general in $\{\pm 1\}\times \Z_2$ and so the description of the action of $G_F(2)$ on the roots of unity would be more complicated. As a consequence, the Elementary Type Conjecture on the level of groups, which relies on the aforementioned group action, would be more complicated as well (see \cite{jacwar}). We can avoid this more complicated case here because in Section \ref{sec:unconditional} we consider a related but distinct Elementary Type Conjecture, on the level of Witt rings of quadratic forms over $F$ rather than on the level of $G_F(p)$. This conjecture is sufficient to describe \emph{all} known Galois cohomology algebras $H^*(G_F(2),\F_2)$, including the case $p=2$ and $\sqrt{-1}\in F$.

As we noted above, it is not known whether all Demushkin groups are of the form $G_F(p)$. However, we are able to prove our results on Koszulity of their cohomology regardless of whether they are realisable or not. Hence our theorems are formulated in terms of the following class of groups, which may possibly be larger than necessary from a Galois-theoretic perspective.

\begin{defn}\label{defi:ETC}
The class $\mathcal{E}_p$ of \emph{elementary type} cyclotomic pairs is the smallest class of cyclotomic pairs such that
\begin{enumerate}
 \item[(a)] any pair $(S,\chi)$, with $S$ a finitely generated free pro-$p$ group and $\chi$ an arbitrary cyclotomic character,
 is in $\mathcal{E}_p$;  the pair $(1,1)$ consisting of the trivial group and the trivial character is not excluded;
 \item[(b)] any pair $(G,\chi)$, with $G$ a Demushkin group and $\chi$ its unique cyclotomic character, provided $\im\ \chi\subseteq1+p\Z_p$ if $p$ is odd, or $\im\ \chi\subseteq1+4\Z_2$ if $p=2$ (this condition on the cyclotomic character is the translation of \eqref{eq:good cyclochar} in the abstract group-theoretic language, and it implies that the invariant $q$ of $G$ is not 2);
\item[(c)] if $(G_1,\chi_1),(G_2,\chi_2)\in\mathcal{E}_p$, then the free product
$(G_1\ast_pG_2,\chi_1\ast_p\chi_2)$ is also in $\mathcal{E}_p$;
\item[(d)] if $(G,\chi)\in\mathcal{E}_p$, then, for any positive integer $m$, the cyclotomic semi-direct product
$(\Z_p^m\rtimes G,\chi\circ\pi)$ is also in $\mathcal{E}_p$. 
\end{enumerate}
% 
% , provided $\im\chi\subseteq 1+p\Z_p$, is in $\mathcal{E}_p$
% , and all finitely generated pro-$2$ groups of the form $G_F(2)$, for some field $F$ containing $\sqrt{-1}$,
% 
% The additional condition on cyclotomic characters in (b) imitates the situation in which $G=G_F(2)$ with $\sqrt{-1}\in F$.
An \emph{elementary type} pro-$p$ group (\emph{ET group} for short) is a group $G$ appearing in a pair in $\mathcal{E}_p$.
\end{defn}
The \textbf{general Elementary Type Conjecture for pro-$p$ groups} states that all finitely generated pro-$p$ groups of the form $G_F(p)$, for $p$ odd and $F$ a field containing a primitive $p$\textsuperscript{th} root of unity, or for $p=2$ and $\sqrt{-1}\in F$, are ET groups.
Additional details may be found in \cite{Koszul1}.% In the case $p=2$, we require in addition that $H^\bullet(G,\F_2)$ contains a distinguished element $t$ such that, for any $x\in H^1(G,\F_2)$, $x^2=tx$.

There is a parallel inductive description of the Galois cohomology algebras of ET groups in terms of two basic operations.
\begin{defn}\label{def:operations on algebras}
Let $A=Q(V_A, R_A)$ and $B=Q(V_B,R_B)$ be two quadratic algebras over a field $K$.
\begin{enumerate}
\item The \emph{direct sum} of $A$ and $B$ is the quadratic algebra
\[ A\sqcap B=\frac{T(A_1\oplus B_1)}{(R)},\quad 
\text{with }R=R_A\oplus R_B\oplus(A_1\otimes B_1)\oplus(B_1\otimes A_1). \]
\item The \emph{skew-symmetric tensor product} of $A$ and $B$ is the quadratic algebra
\[A\otimes^{-1} B=\frac{T(A_1\oplus B_1)}{(R)},\quad \text{with }R=R_A\oplus R_B\oplus \gen{ab+ba\mid a\in A_1,b\in B_1}.\]
\item Suppose that $A$ has a distinguished element $t\in A_1$ such that $t+t=0$, and let $X_J=\{x_j\mid j\in J\}$ be a set of distinct symbols not in $A$. The \emph{twisted extension} of $A$ by $X_J$ over $t$ is the quadratic $K$-algebra $A(t\mid X_J)$ with space of generators $\Span_{K}(A_1\cup\{x_j\mid j\in J\})$ and space of relators \[\Span_{K}(R\cup \{x_ix_j+x_jx_i,x_ja+ax_j,x_j^2-tx_j\mid i,j\in J, a\in A_1\})\]. When $X_J=\{x_1,\dots,x_m\}$, we will use the notation $A(t\mid x_1,\dots,x_m)$.
%where $I$ is the subspace generated by the elements $$ with $$ and $$.
\end{enumerate}
\end{defn}
\begin{rem}\label{rem:twisted ext varia}
Note that the condition $t+t=0$ in the definition of twisted extension forces $t=0$ whenever the characteristic of the ground field differs from $2$. If the characteristic is $2$, then $t$ may or may not be $0$.
A twisted extension $A(0\mid X_J)$ with $t=0$ is the same as the skew-symmetric tensor product of $A$ with the exterior algebra $\Lambda \gen{X_J}$.
Moreover, $A(0\mid x_1,\dots,x_m)$ admits an additional $\oplus_{i=1}^m\Z$-grading, characterised by
\begin{equation}\label{eq:x-grading}
A(0\mid x_1,\dots,x_m)_{(\varepsilon_1,\dots,\varepsilon_m)}=\Span_{K}\{ax_1^{\varepsilon_1}\cdots x_m^{\varepsilon_m}\mid a\in A\}.
\end{equation}
The defining relations make the algebra concentrated in degrees $\{0,1\}^m$.

In particular, $A(0\mid x)$ has the additional $\Z$-grading 
\[
A(0\mid x)_{0}=A, A(0\mid x)_{1}=Ax, A(0\mid x)_{k}=0 \mbox{ for }k\geq 2.
\]
Independently of whether or not $t=0$, the elements of a twisted extension $A(t\mid x)$ can be written in the \emph{normal form} $u+vx$ for suitable $u,v\in A$. By the shape of the defining relations of $A(t\mid x)$, a monomial in the left ideal $(x)$ can be written as a word with only one $x$, but it is not possible to write it as a word without $x$. As a consequence, the coefficients $u$ and $v$ above are uniquely determined: if $u+vx=u'+v'x$, then $u-u'=(v'-v)x$; since the left-hand-side belongs to $A$, while a multiple of $x$ belongs to $A$ if and only if it is $0$, we conclude that $v=v'$ and $u=u'$.
\end{rem}
If $G_1$ and $G_2$ are pro-$p$ groups, then $H^\bullet(G_1\ast_p G_2,\F_p)\cong H^\bullet(G_1,\F_p)\sqcap H^\bullet(G_2,\F_p)$ (see \cite[Theorem 4.1.4]{NSW}), and if $G$ is an ET pro-$p$ group, then $H^\bullet(\Z_p^m\rtimes G,\F_p)\cong H^\bullet(G,\F_p)(0\mid x_1,\dots,x_m)$ (see \cite[Section 5.4]{Koszul1}). Therefore, the cohomology $H^\bullet(G,\F_p)$ of an ET group can be obtained via a finite sequence of direct sums and skew-symmetric tensor products of cohomology algebras with $\F_p$ coefficients of free pro-$p$ groups and cohomology algebras with $\F_p$ coefficients of ET Demushkin pro-$p$ groups.

The results in Section \ref{ssec:witt} imply an analogous description of the Galois cohomology algebra of finitely generated pro-$2$ groups of the form $G_F(2)$ which may not necessarily satisfy Definition \ref{defi:ETC}. One has simply to allow general twisted extensions over an element $t$ that is not necessarily $0$ (see Proposition \ref{prop:graded witt ring group construction}). Through this idea, for pro-$2$ groups we may formulate a weaker form of Elementary Type Conjecture, which gathers the case of $G_F(2)$ with $\sqrt{-1}\in F$ and the case of $G_F(2)$ with $\sqrt{-1}\notin F$ under the same umbrella.

A \emph{Weak elementary type} (WET) pro-$2$ group is a finitely generated pro-$2$ group $G$ such that $H^\bullet(G,\F_2)$ can be obtained via a finite sequence of direct sums and twisted extensions of cohomology algebras with $\F_2$ coefficients of free pro-$2$ groups and cohomology algebras with $\F_2$ coefficients of Demushkin pro-$2$ groups.
The \textbf{Weak Elementary Type Conjecture for pro-$2$ groups} claims that all finitely generated pro-$2$ groups of the form $G_F(2)$, for some field $F$ of characteristic not $2$, are WET groups.% Of course, this conjecture overlaps with the Elementary Type Conjecture for pro-$p$ groups in the case of $p=2$ and $\sqrt{-1}\in F$.

%To sum up, the separation of $G_F(p)$ between the well-behaved cases, that is $p$ odd or $p=2$ and $\sqrt{-1}\in F$, and the complicated case $p=2$ and $\sqrt{-1}\notin F$, is irrelevant on the level of cohomology, provided we . Since we are interested in the 
To sum up, on the level of cohomology we can reunify the various forms of Elementary Type Conjecture into the description of a single family of algebras which encompasses the Galois cohomology algebras of any known $G_F(p)$ satisfying the Standing Hypothesis \ref{ass}.
%
%The existence of a Koszul flag is a ``stratified'' Koszul property: the augmentation ideal can be sliced into strata in such a way that each stratum is generated by a single element and each partial pile of strata has a linear resolution. This is the weakest form of strong Koszulity, and hence the most extensively applicable, but it has a strong drawback: the need to find linear resolutions of many ideals.
% \begin{defn}
% Let $\mathcal F$ be a Koszul filtration of $A$. Then each ideal $I \in \mathcal F$ is a Koszul module.
% \end{defn}
%
%
\section{Universal Koszulity of ET groups}\label{sec:uK}
\subsection{Free pro-$p$ groups}\label{ssec:uK free}
\begin{prop}\label{prop:free uK}
Suppose that $G$ is a finitely generated free pro-$p$ group. Then $H=H^\bullet(G,\F_p)$ is universally Koszul.
\end{prop}
\begin{proof}
Since the cohomology is concentrated in degrees $0$ and $1$, the product of any two elements of positive degree is $0$. Hence, $I:x=H_+$ for all $I\id H, I\neq H_+$ and for all $x\in H_1\setminus I$. The claim follows from condition \ref{cond3} of Proposition \ref{prop:equiv to uK}.
\end{proof}

\subsection{Demushkin groups}\label{ssec:uK Demushkin}
\begin{prop}\label{prop:C_2 uK}
The cohomology $H=H^\bullet(C_2,\F_2)$ of the group of order $2$ is universally Koszul.
\end{prop}
\begin{proof}
As recalled in Remark \ref{rem:cohomology of Demushkin}, $H=\F_2[t]$ is a free algebra on one generator. Therefore, $I:x=(0)$ for all $I\in\mathcal{L}(H)\setminus\{H_+\}$ and all $x\in H_1\setminus I$. The claim follows from condition \ref{cond3} of Proposition \ref{prop:equiv to uK}.
\end{proof}

\begin{prop}\label{prop:demushkin uK}
Suppose that $G$ is an infinite Demushkin pro-$p$ group. Then $H=H^\bullet(G,\F_p)$ is universally Koszul.
\end{prop}
\begin{proof}
Let $I\in\mathcal{L}(H)\setminus\{H_+\}$ and let $x\in H_1\setminus I$. Let us first address the case $I=(0)$. The ideal $(0):x$ is made of all of the solutions of the equation $ax=0$ in the variable $a\in H$. This equation has a nonzero solution in $H_1$ for any $x$, by the definition of a Demushkin group (see also Remark \ref{rem:cohomology of Demushkin}). So $(0):x\neq(0)$, whence $H_1((0):x)\supseteq H_2$ by nondegeneracy of the cup product. Now, for a general $I$, we have $(0):x\subseteq I:x$, so again $H_1(I:x)\supseteq H_2$.

But in general, if $J$ is a left ideal of a graded algebra $A$ generated in degree $1$ such that $A_1J\supseteq A_2$, then $J=AJ_1$. In fact, using the hypothesis that $A$ is generated in degree $1$, for all $n\in\N$,
\[
A_{n+1}=A_{n-1}A_2\subseteq A_{n-1}A_1J=A_nJ.
\]
Now if $a$ is a homogeneous element of $J_{n+1}$ for some $n\in\N$, then $a\in A_{n+1}\subseteq A_nJ$, and so, taking the degree of $a$ into account, $a\in A_nJ_1$.

This shows that condition \ref{cond3} of Proposition \ref{prop:equiv to uK} holds.
\end{proof}

\subsection{Direct sum}
\begin{prop}\label{prop:direct sum universally Koszul}
Let $A$ and $B$ be universally Koszul algebras over an arbitrary field $K$. Then the direct sum $C=A\sqcap B$ is universally Koszul.
\end{prop}
\begin{proof}
The proof is almost exactly the same as in the commutative case (\cite[Lemma 1.6(3)]{conca:u-Koszul}).
Let $I\in\mathcal{L}(C)\setminus \{C_+\}$ and let $x\in C_1=A_1\oplus B_1\setminus I$. Write $I=(a_1+b_1,\dots,a_k+b_k)$ and $x=a+b$, with $a_i,a\in A_1$ and $b_i, b\in B_1$. Set $J_A=(a_1,\dots,a_k)$, $J_B=(b_1,\dots,b_k)$ and $J=J_A\!\,^e+J_B^e=J_A+J_B\id C$. Then $I\subseteq J$. Moreover, if $c\in J_i$ for $i\geq 2$, then $c\in I_i$ because $A_+$ annihilates $B_+$ and conversely.
We claim that $I:x=(J:x)\cap C_+$. In fact, on the one hand, if $c\in C$ is such that $cx\in I$, then $cx\in J$, so $c\in J:x$, and $c\in C_+$ since by assumption $x\notin I$. On the other hand, $c\in J:x$ and $c\in C_+$ imply $cx\in\oplus_{i\geq 2}J_i=\oplus_{i\geq 2}I_i$, so $c\in I:x$.

As a consequence, if $x\in J$, then $I:x=C_+\in\mathcal L(C)$. Otherwise, $I:x=J:x$, and there are three cases to consider.
\begin{enumerate}
\item If $a\notin J_A$ and $b\notin J_B$, then $J:x=(J_A:_Aa)^e+(J_B:_Bb)^e$.
Indeed, let $c\in J:x$. Using Remark \ref{rem:homogeneous colon ideals}, we may assume that $c$ is homogeneous, and $c\in C_+$ since $x\notin J$. Then $cx\in J_A+J_B$. Writing $c=c_A+c_B$ with $c_A\in A_+$ and $c_B\in B_+$, we obtain $cx=(c_A+c_B)a+(c_A+c_B)b=c_Aa+c_Bb\in J_A+J_B$. Hence, $c_Aa\in J_A$ and $c_Bb\in J_B$. Therefore, $c_A\in J_A:_Aa$ and $c_B\in J_B:_Bb$, whence $c\in(J_A:_Aa)^e+(J_B:_Bb)^e$. Conversely, if $c=c_A+c_B\in(J_A:_Aa)^e+(J_B:_Bb)^e$, then $c_A\in A_+$ and $c_B\in B_+$, as $a\notin J_A$ and $b\notin J_B$. Hence, $cx=c_Aa+c_Bb$ with $c_Aa=ca\in J_A$ and $c_Bb=cb\in J_B$. Therefore, $cx\in J_A+J_B$, or equivalently, $c\in J:x$.
\item If $a\in J_A$ and $b\notin J_B$, then $J:x=(A_+)^e+(J_B:_Bb)^e$. In fact, if $c\in J:x$, we may again assume that $c$ is homogeneous, by Remark \ref{rem:homogeneous colon ideals}, and that $c\in C_+$, because $x\notin J$. Then with the same notation as before, $cx=c_Aa+c_Bb\in J_A+J_B$. Since $a\in J_A$, for all $c_A\in A_+$, $c_Aa\in J_A$, hence $c_Bb\in (J_A+J_B)\cap B=J_B$. To sum up, $c\in (A_+)^e+(J_B:_Bb)^e$. The reverse inclusion is obvious.
\item If $a\notin J_A$ and $b\in J_B$, then $J:x=(J_A:_Aa)^e+(B_+)^e$. The proof is analogous to the previous case.
\end{enumerate}
In all cases, since $J_A:_Aa\in\mathcal{L}(A)\ni A_+$ and $J_B:_Bb\in\mathcal L(B)\ni B_+$, we deduce that $J:x\in\mathcal L(C)$. The statement now follows from condition \ref{cond3} of Proposition \ref{prop:equiv to uK}.
\end{proof}

\subsection{Twisted extension}
\begin{prop}\label{prop:twisted extension universally Koszul}
Let $A=Q(V,R)$ be a universally Koszul algebra over an arbitrary field $K$. Then the twisted extension $B:=A(t\mid x_1,\dots,x_m)$ is universally Koszul.
\end{prop}
\begin{proof}
Since $A(t\mid x_1,\dots,x_m)=A(t\mid x_1)(t\mid x_2)\dots(t\mid x_m)$, by induction it is enough to prove the claim for 
\[
B=A(t\mid x)=\frac{T(V\oplus\Span_{K}\{x\})}{(R\cup\{x^2-tx,xa+ax\mid a\in V\})}.
\]
The collection $\mathcal L(B)$ obviously satisfies the first two conditions in the definition of Koszul filtration, so we shall focus on the third.
Let $I\in\mathcal L(B)\setminus\{(0)\}$.
\begin{enumerate}
\item Suppose that $I$ has a complete set of generators belonging to $A$; that is, $I=I_A\!\,^e$ for some $I_A\in\mathcal L(A)$. Then, by the universal Koszulity of $A$, there exist $J_A\in\mathcal{L}(A)\setminus\{I_A\}$ and $a\in A_1$ such that $I_A=J_A+Aa$ and $J_A:_Aa\in\mathcal L(A)$. Set $J=J_A\!\,^e$.

First we claim that $J\cap A=J_A$. Indeed, let $\{\widetilde{a_i}\}$ be a set of generators of $J_A$. The same set also generates $J$ as a left ideal of $B$. Let $u\in J\cap A$; then $u=\sum_i(u_i+v_ix)\widetilde{a_i}=\sum_iu_i\widetilde{a_i}-\left(\sum_iv\widetilde{a_i}\right)x$; since $u\in A$, it follows from the uniqueness of the normal form of elements in $B$ that $u=\sum_iu_i\widetilde{a_i}$, whence $u\in J_A$. The reverse inclusion is obvious.

Then we claim that $J\neq I=J+Ba$. The equality is a direct consequence of $I_A=J_A+Aa$. As regards the inequality, suppose by contradiction that $J=I$; then $a\in J$; but since $a\in A$, by the previous claim, $a\in J_A$; as a consequence, $I_A=J_A+Aa=J_A$, in contradiction to the construction of $J_A$.

Finally, we claim that $J:_Ba=(J_A:_Aa)^e$. In fact, if $J:_Ba$ contains an arbitrary element in normal form $u+vx\in B, u,v\in A$ (recall Remark \ref{rem:twisted ext varia}), then $ua+vxa=ua-vax\in J$. Writing $ua-vax=\sum_i(u_i+v_ix)\widetilde{a_i}$, where the $\widetilde{a_i}$'s are the generators of $J$, and factoring $x$ out in the right-hand-side, by the uniqueness of the normal form $ua=\sum_iu_i\widetilde{a_i}\in J\cap A$ and $va=\sum_iv_i\widetilde{a_i}\in J\cap A$ separately. This means that $u,v\in J_A:_Aa$, whence $u+vx\in(J_A:_Aa)^e$. For the reverse inclusion, the generators of $J_A:_Aa$ belong to $J:_Ba$ by construction. This implies that $J:_Ba\in\mathcal{L}(B)$.

\item Now suppose that $I$ does not have a set of generators all included in $A$. Using suitable $K$-linear combinations of generators, which of course preserve degrees, we can reduce to the case $I=I_A\!\,^e+(x+l)$, where $I_A\in\mathcal L(A)$ and $l\in A_1\setminus I$. We first address the case $l=0$, that is, $I=I_A\!\,^e+Bx$. Set $J=I_A\!\,^e$. Then $J\neq I=J+Bx$ and $J:_Bx={J}+(-t+x)$. In fact, if $J:_Bx$ contains an arbitrary element $u+vx\in B$ in normal form as before, then $ux+vx^2=ux+vtx\in J$. Writing $ux+vtx=\sum_i(u_i+v_ix)\widetilde{a_i}$, where the $\widetilde{a_i}$'s are the generators of $J$, by the uniqueness of the normal form $u+vt=-\sum_iv_i\widetilde{a_i}\in J$. But then $u+vx=u+vt+v(-t+x)\in J+(-t+x)$. For the reverse inclusion, the generators of $J+(-t+x)$ belong to $J:_Bx$ thanks to the defining relation $x^2-tx$ of $B$. Again it follows that $J:_Bx\in\mathcal{L}(B)$.

Now if $l$ is arbitrary, we consider the automorphism of graded algebras\footnote{The symbol $\squig$ was chosen and strongly defended by the second author, despite the doubts of the third author.} $\squig:B\to B$ defined as $\squig(u+vx)=u+v(x-l)$. This automorphism sends $J$ to $I^e+Bx$, so the general case reduces to the previous special case.
\end{enumerate}
The statement now follows from condition \ref{cond4} of Proposition \ref{prop:equiv to uK}.
\end{proof}

\section{Strong Koszulity of ET groups}
\subsection{Free pro-$p$ groups}\label{ssec:sK free}
\begin{prop}\label{prop:free sK}
Suppose that $G$ is a free pro-$p$ group on $n$ generators. Then $H=H^\bullet(G,\F_p)$ is strongly Koszul.
\end{prop}
\begin{proof}
In this case, the cohomology is concentrated in degrees $0$ and $1$, that is, the space of relators of the cohomology is the span of all quadratic monomials. The result then follows from \cite[Corollary 6.6]{eneHerzog}.
\end{proof}

\subsection{Demushkin groups}\label{ssec:sK Demushkin}
Let $d$ denote the cardinality of a minimal set of generators of a Demushkin group $G$.
\begin{prop}\label{prop:c_2 sK}
Suppose that $G=C_2$. Then $H=H^\bullet(G,\F_2)$ is strongly Koszul.
\end{prop}
\begin{proof}
As recalled in Remark \ref{rem:cohomology of Demushkin}, $H=\F_2[t]$, which is strongly Koszul, because $(0):t=(0)$.
\end{proof}

\begin{prop}\label{prop:demushkin sK 1}
Suppose that $G$ is a Demushkin group with invariant $q\neq2$. Then $H=H^\bullet(G,\F_p)$ is strongly Koszul.
\end{prop}
\begin{proof}
In this case, there exists a symplectic basis $X=\{a_1,\dots,a_{2k}\}$ of $H_1=H^1(G,\F_p)$ as in case $1$ of Remark \ref{rem:cohomology of Demushkin}.
Now there are only two cases to consider in order to verify the condition of Definition \ref{def:strongly koszul}.
\begin{enumerate}
\item If the subset $Y$ is empty, then the only colon ideal involved is 
\[
(0)\colon a_i=\begin{cases}
              (\{a_j\mid j\neq i-1\}) & i\mbox{ even}\\
              (\{a_j\mid j\neq i+1\}) & i\mbox{ odd}.
              \end{cases}
\]
\item If $Y$ is non-empty, then the colon ideals of the form
$(a_{i_1},\dots,a_{i_{j-1}}):a_{i_j}$ for $j>1$ are involved. But each of these ideals is just the whole $(X)$.
\end{enumerate}
\end{proof}
\begin{prop}
Suppose that $G$ is a Demushkin pro-$2$ group on $d=2k+1$ generators, $k\in\N\setminus\{0\}$. Then $H=H^\bullet(G,\F_2)$ is strongly Koszul.
\end{prop}
\begin{proof}
There is a basis $X=\{a_1,\dots,a_{2k+1}\}$ of $H_1=H^1(G,\F_2)$ as in case $2$ of Remark \ref{rem:cohomology of Demushkin}.
\begin{enumerate}
\item If the subset $Y$ is empty, then the only colon ideal involved is 
\[
(0)\colon a_i=\begin{cases}
              (\{a_j\mid j\neq 1\}) & i=1\\
              (\{a_j\mid j\neq i+1\}) & i\mbox{ even}\\
              (\{a_j\mid j\neq i-1\}) & i>1\mbox{ odd}.
              \end{cases}
\]
\item If $Y$ is non-empty, then our argument is exactly as in Proposition \ref{prop:demushkin sK 1}.
\end{enumerate}
\end{proof}
\begin{prop}
Suppose that $G$ is a Demushkin pro-$2$ group on $d=2k$ generators, $k\in\N\setminus\{0\}$, with invariant $q=2$. Then $H=H^\bullet(G,\F_2)$ is strongly Koszul.
\end{prop}
\begin{proof}
There is a basis $X=\{a_1,\dots,a_{2k}\}$ of $H_1=H^1(G,\F_2)$ as in case $3$ of Remark \ref{rem:cohomology of Demushkin}.
With the change of basis
\[
b_1=a_1,\quad b_2=a_2+a_1, \quad b_i=a_i\; (i=3,\dots,2k),
\]
the cup product is given by
\[
\begin{array}{l}
b_1\cup b_1=b_2\cup b_2=b_3\cup b_4=\dots=b_{2k-1}\cup b_{2k}=1,\\
b_4\cup b_3=\dots=b_{2k}\cup b_{2k-1}=1,\\
b_i\cup b_j=0 \mbox{ for all other }i,j.
\end{array}
\]
This provides the following description of the relevant colon ideals.
\begin{enumerate}
\item If the subset $Y$ is empty, then the only colon ideal involved is 
\[
(0)\colon b_i=\begin{cases}
              (\{b_j\mid j\neq 1\}) & i=1\\
              (\{b_j\mid j\neq 2\}) & j=2\\
              (\{b_j\mid j\neq i+1\}) & i>1\mbox{ odd}\\
              (\{b_j\mid j\neq i-1\}) & i>2\mbox{ even}.
              \end{cases}
\]
\item If $Y$ is non-empty, then our argument is exactly as in Proposition \ref{prop:demushkin sK 1}.
\end{enumerate}
\end{proof}
\subsection{Direct sum}
\begin{prop}\label{prop:direct sum strongly Koszul}
Let $A$ and $B$ be strongly Koszul algebras over an arbitrary field $K$ with respect to the minimal systems of homogeneous generators $X_A=\{a_1,\dots,a_c\}$ and $X_B=\{b_1,\dots,b_d\}$, respectively. Then the direct sum $A\sqcap B$ is strongly Koszul with respect to the minimal system of generators $\{a_1,\dots,a_c,b_1,\dots,b_d\}$.
\end{prop}
\begin{proof}
The two key ideas are that every element $x\in A\sqcap B$ has a unique decomposition $x=x_{A} + x_{B}$, with $x_{A}\in A$ and $x_{B}\in B$, and that the elements of $A$ annihilate $x_{B}$, and conversely.
\begin{enumerate}
\item A left ideal of $A\sqcap B$ of the shape $I=(a_{i_1},\dots,a_{i_{k-1}}):a_{i_{k}}$ coincides with $(I^c_A)^e+(B_+)^e$ by the definition of the direct sum. By hypothesis, $I^c_A$ is generated, as a left ideal of $A$, by a subset $Y_A$ of $X_A$, so $I=(Y_A\cup X_B)$. An analogous argument works for a left ideal of $A\sqcap B$ of the shape $(b_{i_1},\dots,b_{i_{k-1}}):b_{i_{k}}$.
\item A left ideal of $A\sqcap B$ of the shape $I=(a_{i_1},\dots,a_{i_{h}},b_{j_1},\dots,b_{j_{k-1}}):b_{j_{k}}$ coincides with $
(A_+)^e+(((b_{j_1},\dots,b_{j_{k-1}}):b_{j_{k}})^c_B)^e$ by the definition of the direct sum. By hypothesis, $((b_{j_1},\dots,b_{j_{k-1}}):b_{j_{k}})^c_B$ is generated, as a left ideal of $B$, by a subset $Y_B$ of $X_B$, so $I=(X_A\cup Y_B)$. An analogous argument works for a left ideal of $A\sqcap B$ of the shape $I=(a_{i_1},\dots,a_{i_{h-1}},b_{j_1},\dots,b_{j_{k}}):a_{i_{h}}$.
\end{enumerate}
\end{proof}

\subsection{Twisted extension}
\begin{prop}\label{prop:twisted extension strongly Koszul}
Let $A=Q(V,R)$ be a strongly Koszul $K$-algebra over an arbitrary field $K$ with respect to the minimal system of homogeneous generators $X_A=\{a_1,\dots,a_d\}$ and let $X_J=\{x_1,\dots,x_m\}$. Then the twisted extension $B:=A(0\mid X_J)=A\otimes^{-1}\Lambda(x_1,\dots,x_m)$ is strongly Koszul with respect to the minimal system of homogeneous generators $\{a_1,\dots,a_d,x_1,\dots,x_m\}$.
\end{prop}
\begin{proof}
Since $A(0\mid X_J)=A(0\mid x_1)(0\mid x_2)\dots(0\mid x_m)$, by induction it is enough to prove the statement for the algebra
\[
B=A(0\mid x)=\frac{T(\Span_{K}(X_A\cup \{x\})}{(R\cup\{x^2,xa_i+a_ix\mid a_i\in X_A\})}.
\]
% Since $A[0,J]=((A[0,\{1\}])[0,\{2\}])\dots[0,\{m\}]$, by induction it is enough to prove the claim for $J=\{1\}$, that is, for the algebra
% \[
% B=A[0,\{1\}]=\frac{T(V\oplus x)}{(R\cup\{x^2,xa+ax\mid a\in V\})}.
% \]
The only relation between $x$ and any element of $A_+$ is skew-commutativity. As a consequence:
\begin{enumerate}
\item If, in $A$, $(a_{i_1},\dots,a_{i_{k}}):a_{i_{k+1}}=(a_{j_1},\dots,a_{j_{r}})$, then in $B$
\[\begin{array}{l}
(a_{i_1},\dots,a_{i_{k}}):a_{i_{k+1}}=(a_{j_1},\dots,a_{j_{r}}),\\
(a_{i_1},\dots,a_{i_{k}}, x):a_{i_{k+1}}=(a_{j_1},\dots,a_{j_{r}},x).
\end{array}
\]
\item $(a_{i_1},\dots,a_{i_{k}}):x=(a_{i_1},\dots,a_{i_{k}},x)$.
\end{enumerate}
% In fact, the left-hand-side colon ideals are homogeneous with respect to the additional grading \eqref{eq:x-grading}. Moreover, any element $b\in B$ is a sum $p+xq$ with $p,q\in A$ and the summands $p$ and $xq$ are exactly the homogeneous components of $b$ with respect to the additional grading \eqref{eq:x-grading}. Hence $b$ belongs to a homogeneous ideal of $B$ if and only if both $p$ and $xq$ belong to the ideal.
In fact any element $b\in B$ can be expressed as a sum $u+vx$ with $u,v\in A$, and the summands $u$ and $vx$ are exactly the homogeneous components of $b$ with respect to the additional grading \eqref{eq:x-grading}.
Since the left-hand-side colon ideals are homogeneous with respect to the additional grading \eqref{eq:x-grading}, an element $b$ belongs to one of these ideals, if and only if both $u$ and $vx$ belong to that ideal. %But then, an element of the form $pa$, $a\in A$, belongs to an ideal of $B$ generated by elements of $A$ if and only if $pa$ belongs, as an element of $A$, to the ideal of $A$ generated by the same generators. An element of the form $xqa$, $a\in A$, belongs to an ideal of $B$ generated by elements of $A$ if and only if $qa$ belongs to that ideal. Finally, an element of the form $px$ belongs to an ideal of $B$ generated by elements of $A$ if and only if $p$ belongs to that ideal, $px$ always belongs to an ideal that has $x$ as a generator and $vxx=0$.
\end{proof}

\section{Exceptions to Strong Koszulity}\label{sec:baddies}
Strong Koszulity may not be preserved by twisted extensions in characteristic $2$, when $t\neq 0$. In this section we present two notable families of examples.
\begin{defn}[\cite{lam2005}]
The \emph{level} (\emph{Stufe}) of a field $F$, denoted $s(F)$, is the smallest positive integer $k$ such that $-1$ can be written as a sum of $k$ squares in $F$, provided such integer exists, and $0$ otherwise.
\end{defn}
\begin{rem}
If $\chr F=2$, then $s(F)=1$, since $-1$ is a square.
If $\chr F\neq2$, then $s(F)=0$ if and only if the field is formally real. Otherwise, $s(F)$ is a power of $2$ (see \cite{lam2005}).

The level of a field $F$ is related to the value set of some quadratic forms over it. For $a\in F^\times$, $[a]$ denotes the image of $a$ in $F^\times/(F^\times)^2$. The \emph{value set} of a quadratic form $\phi$ over $F$ is defined as $D_F\phi=\{[a]\in F^\times/(F^\times)^2\mid a\mbox{ is represented by }\phi\}$.
\end{rem}

\begin{defn}
A field $F$ is \emph{$2$-rigid} if, for all $a\in F$ such that $[a]\neq 1$ and $[a]\neq[-1]$, the value set of the quadratic form $\gen{1,a}=X^2+aY^2$ is included in $\{[1],[a]\}\subseteq F^\times/(F^\times)^2$.
\end{defn}

By \cite[Theorem 1.9, Proposition 1.1]{ware}, the level of a $2$-rigid field $F$ is either $0$, $1$ or $2$. The case $s(F)=1$ corresponds to the condition that $F$ contains a square root of $-1$. A $2$-rigid field is \emph{superpythagorean} if it is of level $0$ %is called \emph{superpythagorean}.

\subsection{Superpythagorean fields}\label{ssec:superpythagorean}
Let $F$ be a superpythagorean field such that $\dim_{\F_2}F^\times/(F^\times)^2=d<\infty$ and let $\{[-1],[a_2],\dots,[a_d]\}$ be a fixed $\F_2$-basis of $F^\times/(F^\times)^2$. Then a presentation of $H=H^\bullet(G_F(2),\F_2)$ is
\begin{equation}\label{eq:cohom superpyth}
H=\F_2\gen{t,\alpha_2,\dots,\alpha_n\mid \alpha_j\alpha_i=\alpha_i\alpha_j, \alpha_it=t\alpha_i,\alpha_i\alpha_i=t\alpha_i},
\end{equation}
where $t=\ell([-1])$, $\alpha_i=\ell([a_i])$ and cup products are omitted (see \cite{wadsworth:cohomology}).
For all $n\geq 1$, the set
\[
B_n=\{t^{n-r}\alpha_{i_1}\dots\alpha_{i_r}\mid 0\leq r\leq d,\quad i_1<\dots <i_r\}
\]
forms a $\F_2$-basis of $H^n(G_F(2),\F_2)$ (see \cite[Theorem 5.13(2)]{elmanLam} and \cite[Corollary 7.5]{voe}).

\begin{prop}\label{prop:superpythagorean PBW}
If $F$ is a superpythagorean field with $\dim_{\F_2}F^\times/(F^\times)^2<\infty$, then $H=H^\bullet(G_F(2),\F_2)$ is PBW.
\end{prop}
\begin{proof}
We keep the previous notation and we apply the Rewriting Method of \cite[Section 4.1]{lodval}.
Consider the degree-lexicographic order on the monomials of $T\gen{t,\alpha_2,\dots,\alpha_d}$ induced by choosing the total order $t<\alpha_2<\dots<\alpha_d$. Then a normalized basis for the space of relations of $H$ is
\[
\{\alpha_j\alpha_i-\alpha_i\alpha_j\mid2\leq i<j\leq d\}\cup\{\alpha_it-t\alpha_i\mid2\leq i\leq d\}\cup\{\alpha_i\alpha_i-t\alpha_i\mid2\leq i\leq d\}.
\]
The corresponding critical monomials are
\begin{enumerate}
\item $\alpha_i\alpha_i\alpha_i$ $(2\leq i\leq d)$;
\item $\alpha_j\alpha_j\alpha_i$ $(2\leq i<j\leq d)$;
\item $\alpha_j\alpha_i\alpha_i$ $(2\leq i<j\leq d)$;
\item $\alpha_i\alpha_it$ $(2\leq i\leq d)$;
\item $\alpha_k\alpha_j\alpha_i$ $(2\leq i<j<k\leq d)$;
\item $\alpha_j\alpha_it$ $(2\leq i<j\leq d)$.
\end{enumerate}
The graphs of the critical monomials are
\begin{multicols}{2}
\begin{description}
\item[Type (1)] $
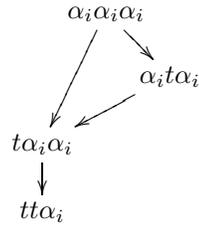
\xymatrix@R-12pt@C-32pt{&\alpha_i\alpha_i\alpha_i\ar[ddl]\ar[dr]&\\
&&\alpha_it\alpha_i\ar[dll]\\
t\alpha_i\alpha_i\ar[d]&&\\
tt\alpha_i}$
\item[Type (2)] $
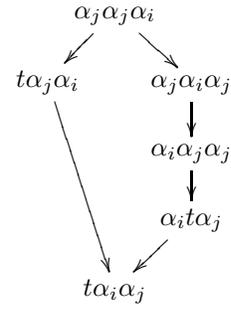
\xymatrix@R-12pt@C-32pt{&\alpha_j\alpha_j\alpha_i\ar[dl]\ar[dr]&\\
t\alpha_j\alpha_i\ar[dddr]&&\alpha_j\alpha_i\alpha_j\ar[d]\\
&&\alpha_i\alpha_j\alpha_j\ar[d]\\
&&\alpha_it\alpha_j\ar[dl]\\
&t\alpha_i\alpha_j&}$
\end{description}
\end{multicols}
\begin{multicols}{2}
\begin{description}
\item[Type (3)] $
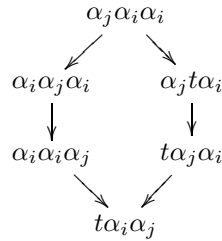
\xymatrix@R-12pt@C-32pt{&\alpha_j\alpha_i\alpha_i\ar[dl]\ar[dr]&\\
\alpha_i\alpha_j\alpha_i\ar[d]&&\alpha_jt\alpha_i\ar[d]\\
\alpha_i\alpha_i\alpha_j\ar[dr]&&t\alpha_j\alpha_i\ar[dl]\\
&t\alpha_i\alpha_j&}$
\item[Type (4)] $
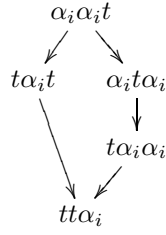
\xymatrix@R-12pt@C-32pt{&\alpha_i\alpha_it\ar[dl]\ar[dr]&\\
t\alpha_it\ar[ddr]&&\alpha_it\alpha_i\ar[d]\\
&&t\alpha_i\alpha_i\ar[dl]\\
&tt\alpha_i&}$
\end{description}
\end{multicols}
\begin{multicols}{2}
\begin{description}
\item[Type (5)] $
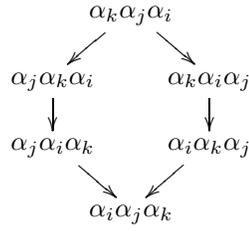
\xymatrix@R-12pt@C-32pt{&\alpha_k\alpha_j\alpha_i\ar[dl]\ar[dr]&\\
\alpha_j\alpha_k\alpha_i\ar[d]&&\alpha_k\alpha_i\alpha_j\ar[d]\\
\alpha_j\alpha_i\alpha_k\ar[dr]&&\alpha_i\alpha_k\alpha_j\ar[dl]\\
&\alpha_i\alpha_j\alpha_k&}$
\item[Type (6)] $
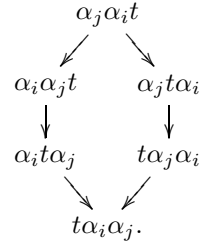
\xymatrix@R-12pt@C-32pt{&\alpha_j\alpha_it\ar[dl]\ar[dr]&\\
\alpha_i\alpha_jt\ar[d]&&\alpha_jt\alpha_i\ar[d]\\
\alpha_it\alpha_j\ar[dr]&&t\alpha_j\alpha_i\ar[dl]\\
&t\alpha_i\alpha_j.&}$
\end{description}
\end{multicols}
Since all critical monomials are confluent, $H$ is PBW.
\end{proof}
\begin{prop}\label{prop:superpythagorean filtration}
If $F$ is a superpythagorean field with $\dim_{\F_2}F^\times/(F^\times)^2<\infty$, then $H=H^\bullet(G_F(2),\F_2)$ is universally Koszul.
\end{prop}
\begin{proof}
Since $H=\F_2[t](t\mid\alpha_2,\dots,\alpha_d)$ is a twisted extension of the universally Koszul algebra $\F_2[t]$, the statement follows from Proposition \ref{prop:twisted extension universally Koszul}.
\end{proof}
\begin{prop}\label{prop:superpythagorean not strongly Koszul}
If $F$ is a superpythagorean field with $3\leq\dim_{\F_2}F^\times/(F^\times)^2<\infty$, then $H=H^\bullet(G_F(2),\F_2)$ is not strongly Koszul.
\end{prop}
\begin{proof}
We first note that, by the shape of the defining relations of $H$, for any $i=1,\dots,d$,
\begin{equation}\label{eq:0 colon a}
(0)\colon\alpha_i=(t+\alpha_i).
\end{equation}
Since the system of generators $\{t,\alpha_1,\dots,\alpha_d\}$ of $H$ is arbitrary, a similar equality holds in complete generality.
Explicitly, for all $a\in F^\times$ such that $[a]\neq[1]$ and $[a]\neq[-1]$, we have that 
\[(0)\colon \ell([a])=\gen{\ell([-a])}.\] 
%In fact, the inclusion $\gen{\ell([-a])}\subseteq(0)\colon \ell([a])$ is a direct consequence of the Steinberg relation. For the vice versa, choose a $\F_2$-basis $\{[-1],[a_2],\dots,[a_d]\}$ of $F^\times/(F^\times)^2$ with $[a_2]=[-a]$. Keeping the notation of Proposition \ref{prop:superpythagorean PBW}, for all $n\geq 1$ the set
%\[
%B_n=\{(-1)^{n-r}\alpha_{i_1}\dots\alpha_{i_r}\mid 0\leq r\leq d,\quad i_1,\dots i_r\mbox{ distinct}\}
%\]
%forms a $\F_2$-basis of $H^n(G_F(2),\F_2)$ (see \cite[Theorem 5.13(2)]{elmanLam} and \cite[Corollary 7.5]{voe}).
%Since $\ell([a])=t+(\alpha_2)$ in $H^1(G_F,\F_2)$, the cup product between $\ell([a])$ and a typical element of $B_n$ is
%\[
%\ell([a])\chi=t^{n-r+1}\alpha_{i_1}\dots\alpha_{i_r} + (\alpha_2)t^{n-r}\alpha_{i_1}\dots\alpha_{i_r}.
%\]
%Thus, there is no way for an element not in $\gen{\ell([-a])}$ to be annihilated by $\ell([a])$.
%%an element $\chi\in H^n$ belongs to $(0)\colon \ell([a])$ if and only if, in the decomposition of $\chi$ as linear combination of $B_n$, only the basis vectors containing the factor $\alpha_2$ have a nonzero coefficient.
%%But this is equivalent to the fact that $\chi\in\gen{\ell([-a])}$.
%The proof of the claim is complete.

Since $\dim_{\F_2}F^\times/(F^\times)^2\geq3$, any basis of $H^1(G_F(2),\F_2)$ contains at least two elements $\ell([a]), \ell([b])$ such that the elements $\ell([-1]), \ell([a]), \ell([b]), \ell([-a]), \ell([-b])$ are all distinct. Now suppose that $H$ be strongly Koszul with respect to a minimal system of homogeneous generators $X=\{u_1,\dots,u_n\}\supseteq\{\ell([a]),\ell([b])\}$. Applying the previous equality to $(0)\colon\ell([a])$ and $(0)\colon\ell([b])$ forces
\[
X\supseteq\{\ell([a]), \ell([b]),\ell([-a]),\ell([-b])\}.
\]
But then the system $X$ is not minimal, as it satisfies the relation
\[
\ell([a])+\ell([-a])+\ell([b])+\ell([-b])=\ell([-1])+\ell([-1])=0 \in H^1(G_F(2),\F_2),
\]
a contradiction.
\end{proof}
\subsection{$2$-Rigid fields of level $2$}\label{ssec:level2}
Let $F$ be a $2$-rigid field of level $2$ such that $\dim_{\F_2}F^\times/(F^\times)^2=d<\infty$ and let $\{[-1],[a_2],\dots,[a_d]\}$ be a fixed $\F_2$-basis of $F^\times/(F^\times)^2$. Then a presentation of $H=H^\bullet(G_F(2),\F_2)$ is
\begin{equation}\label{eq:cohom level2}
H=\F_2\gen{t,\alpha_2,\dots,\alpha_d\mid \alpha_j\alpha_i=\alpha_i\alpha_j, \alpha_it=t\alpha_i,\alpha_i\alpha_i=t\alpha_i, tt=0},
\end{equation}
where $t=\ell([-1])$, $\alpha_i=\ell([a_i])$ and cup products are omitted (see \cite{wadsworth:cohomology}).
This graded algebra is concentrated in degrees $0$ to $d$.
For all $n\leq d$, the set
\[
B_n=\{t^{\delta}\alpha_{i_1}\dots\alpha_{i_n-\delta}\mid \delta=0,1,\quad i_1<\dots <i_n-\delta\}
\]
forms a $\F_2$-basis of $H^n(G_F(2),\F_2)$.

We use the previous notation throughout this subsection, unless otherwise specified.

\begin{prop}\label{prop:level2 PBW}
If $F$ is a $2$-rigid field of level $2$ with $\dim_{\F_2}F^\times/(F^\times)^2<\infty$, then $H=H^\bullet(G_F(2),\F_2)$ is PBW.
\end{prop}
\begin{proof}
Consider the degree-lexicographic order on the monomials of $T\gen{t,\alpha_2,\dots,\alpha_d}$ induced by the total order $t<\alpha_2<\dots<\alpha_d$. Then a normalized basis for the space of relations of $H$ in the sense of \cite[Section 4.1]{lodval} is
\[
\{\alpha_j\alpha_i-\alpha_i\alpha_j,\quad \alpha_it-t\alpha_i,\quad \alpha_i\alpha_i-t\alpha_i \mid2\leq i<j\leq d\}\cup\{tt\}.
\]
The six types of monomials introduced in the proof of Proposition \ref{prop:superpythagorean PBW} are again critical. In addition, there are two new types:
\begin{enumerate}
\item[(1')] $ttt$;
\item[(3')] $\alpha_itt$ $(2\leq i\leq d)$.
\end{enumerate}
The graphs of critical monomials of types (2), (3), (5) and (6) are exactly the same as in Proposition \ref{prop:superpythagorean PBW}. The graphs of types (1) and (4) shall be modified as follows:
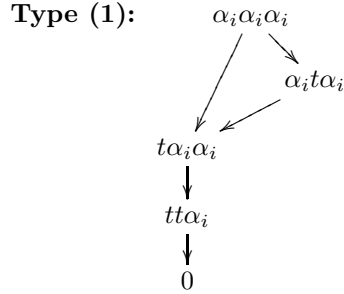
\begin{multicols}{2}
\begin{description}
\item[Type (1)] $\xymatrix@R-12pt@C-32pt{&\alpha_i\alpha_i\alpha_i\ar[ddl]\ar[dr]&\\
&&\alpha_it\alpha_i\ar[dll]\\
t\alpha_i\alpha_i\ar[d]&&\\
tt\alpha_i\ar[d]&&\\
0}$
\item[Type (4)] $
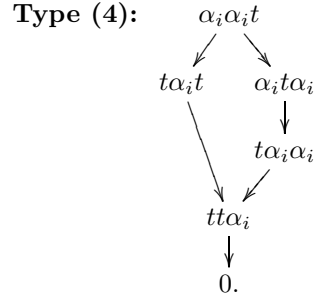
\xymatrix@R-12pt@C-32pt{&\alpha_i\alpha_it\ar[dl]\ar[dr]&\\
t\alpha_it\ar[ddr]&&\alpha_it\alpha_i\ar[d]\\
&&t\alpha_i\alpha_i\ar[dl]\\
&tt\alpha_i\ar[d]&\\
&0.&}$
\end{description}
\end{multicols}
Finally, the graphs of the new critical monomials are
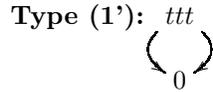
\begin{multicols}{2}
\begin{description}
\item[Type (1')] $\xymatrix@R-12pt@C-32pt{&ttt\ar@/_1pc/[d]\ar@/^1pc/[d]&\\
&0&}$
\item[Type (3')] $
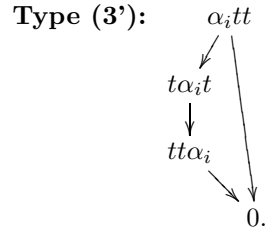
\xymatrix@R-12pt@C-32pt{&\alpha_itt\ar[dl]\ar[dddr]&\\
t\alpha_it\ar[d]\\
tt\alpha_i\ar[drr]\\
&&0.}$
\end{description}
\end{multicols}
Since all critical monomials are confluent, $H$ is PBW.
\end{proof}

\begin{prop}\label{prop:level2 filtration}
If $F$ is a $2$-rigid field of level $2$ with $\dim_{\F_2}F^\times/(F^\times)^2<\infty$, then $H=H^\bullet(G_F(2),\F_2)$ is universally Koszul.
\end{prop}
\begin{proof}
Since $H=\F_2\gen{t\mid t^2}(t\mid\alpha_2,\dots,\alpha_d)$ is a twisted extension of the universally Koszul algebra $\F_2\gen{t\mid t^2}$, the statement follows from Proposition \ref{prop:twisted extension universally Koszul}.
\end{proof}

\begin{prop}\label{prop:level2 not strongly Koszul}
If $F$ is a $2$-rigid field of level $2$ with $3\leq\dim_{\F_2}F^\times/(F^\times)^2<\infty$, then $H=H^\bullet(G_F(2),\F_2)$ is not strongly Koszul.
\end{prop}
\begin{proof}
Identical to that of Proposition \ref{prop:superpythagorean not strongly Koszul}, observing that Equation \eqref{eq:0 colon a} holds in this situation as well.
\end{proof}

\section{Koszul filtrations and ET operations}
It is natural to investigate the compatibility of Koszul filtrations with direct sums and twisted extensions. In view of the sharper results in Section \ref{sec:uK}, the content of this section is not needed for our applications to Galois groups, but it seems to be interesting in its own right.

In this section algebras are defined over an arbitrary field.
\begin{prop}\label{prop:direct sum Koszul filtration}
Let $A$ and $B$ be algebras with respective Koszul filtrations $\mathcal F$ and $\mathcal G$. Then the direct sum $C=A\sqcap B$ has the Koszul filtration $\mathcal H=\mathcal F\sqcap\mathcal G =\{I_A\!\,^e+ I_B\!\,^e\mid I_A\in\mathcal F,I_B\in\mathcal G\}$.
\end{prop}
\begin{proof}
Conditions (1) and (2) of Definition \ref{def:koszul filtration} for the family $\mathcal H$ come directly from the corresponding conditions on $\mathcal F$ and $\mathcal G$.

For Condition (3), we consider two cases.
\begin{enumerate}
\item $I=I_A\!\,^e$, with $I_A\in \mathcal F\setminus\{(0)\}$.\\
Since $A_+$ and $B_+$ annihilate each other, $I=I_A\in\mathcal F$. By hypothesis, there exist $J_A\in\mathcal F\setminus\{I_A\}$ and $a\in A_1$ such that $I_A=J_A+Aa$ and $J_A:_Aa\in\mathcal F$. But, defining $J=J_A\!\,^e\in\mathcal H\setminus\{I\}$, we immediately have $J=J_A$ and $J\neq I=J\!\,^e+Ca$ as left ideals of $C$, and we claim that $J:_Ca=(J_A:_Aa)^e+B_+$. In fact, the inclusion of the right-hand-side into the left-hand-side is clear from the definition of direct sum. Conversely, a typical element of $C$ has the shape $x+y$ for some $x\in A, y\in B$, and without loss of generality we may assume $y\in B_+$. %$\sum_i x_i+\sum_j y_j$ for some $x_i\in A, y_i\in B$, and without loss of generality we may assume all $y_i\in B_+$. let $\sum x_i+y_i\in J:_Ca$ for 
If $x+y\in J:_Ca$, then $xa+ya=xa\in J=J_A$, hence $x\in(J_A:_Aa)^e$, so $x+y\in(J:_Aa)^e+B_+$.
\item $I=I_A\!\,^e+I_B\!\,^e$, with $I_A\in \mathcal F$ and $I_B\in\mathcal G\setminus\{(0)\}$.\\
As before, $I_A\!\,^e=I_A$ and $I_B\!\,^e=I_B$. By hypothesis, there exist $J_B\in\mathcal G\setminus\{I_B\}$ and $b\in B_1$ such that $I_B=J_B+Bb$ and $J_B:_Bb\in\mathcal G$. Set $J=I_A\!\,^e+J_B\!\,^e$. Then $J\neq I=J+Cb$. The claim that $J:_Cb=A_++(J_B:_Bb)^e$ can be proved in an analogous way as before.%But $J=J^e\in\mathcal H\setminus\{I\}$, $I=J+(a)_C$ as ideals of $C$ and we claim that $J:_Ca=(J:_Aa)^e+B_+$.
\end{enumerate}
\end{proof}

\begin{prop}\label{prop:twisted extension Koszul filtration}
Let $A$ be a quadratic algebra with a Koszul filtration $\mathcal F$. Suppose that $t\in A_1$ satisfies $t+t=0$ and that $\mathcal F$ has the property\footnote{The tag \heart\ represents the authors' love for this property.}
\[\label{eq:good filtration}
\tag{\heart} J\in\mathcal F \Rightarrow J+At\in\mathcal{F}.
\]
Then any twisted extension $C=A(t\mid x_1,\dots,x_n)$ has the Koszul filtration 
\[
\mathcal H=\{I^e+(Y)\mid I\in\mathcal F, Y\subseteq\{x_1,\dots,x_m, t-x_1,\dots,t-x_m\}\}.
\]
\end{prop}
\begin{proof}
We first prove the result for the case $C=A(t\mid x), \mathcal{H}=\{I^e+(Y)\mid I\in\mathcal F, Y\subseteq\{x,t-x\}\}$.
Conditions (1) and (2) of Definition \ref{def:koszul filtration} are clearly satisfied. As regards Condition (3), any $c\in C$ can be written in normal form as $c=u+vx$ for some $u,v\in A$ (Remark \ref{rem:twisted ext varia}). We now address several cases separately.
\begin{enumerate}
\item $I=I_A\!\,^e$ for $I_A\in\mathcal{F}\setminus\{(0)\}$.\\
By hypothesis, there are $J_A=({a_1},\dots,{a_r})_A\in\mathcal{F}\setminus\{I_A\}$ and $a\in A_1$ such that $I_A=J_A+Aa$ and $J_A:_Aa\in\mathcal F$. We then set $J=J_A\!\,^e\id C$, so that $J\neq I=J+Ca$, and we claim that $J:_Ca=(J_A:_Aa)^e\in\mathcal H$.
In fact, all the generators of $J_A:_Aa$ belong to $J:_Ca$ by construction. For the reverse inclusion, take any $c=u+vx\in J:_Ca$. Then $ca=ua-vax$ belongs to $J$, that is, $ua-vax=\sum_{i=1}^r(u_i+v_ix){a_i}$ for suitable $u_i,v_i\in A$. By uniqueness of the normal form,
\[
\left\{\begin{array}{l}
ua=\sum_{i=1}^ru_i{a_i}\in J_A,\\
va=\sum_{i=1}^rv_i{a_i}\in J_A.
\end{array}
\right.
\]
But then $u,v\in J_A:_Aa\id A$, hence $u+vx\in (J_A:_Aa)^e\id C$.
\item $I=I_A\!\,^e+Cx$ for $I_A=(a_1,\dots,a_d)_A\in\mathcal{F}$, $d\geq0$ and all $a_i\in A_1$.\\
We set $J=I_A\!\,^e=(a_1,\dots,a_d)_C$, so that $J\neq I=J+Cx$, and we claim that $J:_Cx=I_A\!\,^e+(t-x)_C\in\mathcal H$.
In fact, all the generators of the right-hand-side ideal belong to $J:_Cx$ by construction, using the defining relations of $C$. For the reverse inclusion, take any $c=u+vx\in J:_Cx$. Then $cx=ux+vx^2=ux+vtx=(u+vt)x$ belongs to $J$, that is, $(u+vt)x=\sum_{i=1}^d(u_i+v_ix){a_i}$ for suitable $u_i,v_i\in A$. Since $u+vt\in A$, by uniqueness of the normal form, $u+vt=-\sum_{i=1}^dv_ia_i\in I_A\subseteq I_A\!\,^e$.
Finally, $u+vx=u+vt-v(t-x)\in I_A\!\,^e+(t-x)_C$.
\item $I=I_A\!\,^e+(t-x)_C$ for $I_A\in\mathcal{F}$.\\
This case is completely analogous to the previous one, interchanging the roles of $x$ and $t-x$.
\item $I=I_A\!\,^e+(x,t-x)_C$ for $I_A\in\mathcal{F}$.\\
We can write $I=I_A\!\,^e+Ct+Cx$. Thanks to property \eqref{eq:good filtration}, $I_A+At\in\mathcal F$, so this case is brought back to case 2.
\end{enumerate}
The proof of the particular case $C=A(t\mid x)$ is complete.

Note that $\mathcal{H}$ inherits Property \eqref{eq:good filtration} from $\mathcal{F}$. Therefore, since \[A(t\mid x_1,\dots,x_m)=A(t\mid x_1)(t\mid x_2)\dots(t\mid x_m),\] the general result follows by induction.
\end{proof}
\begin{cor}
Let $A$ be a quadratic algebra with a Koszul filtration $\mathcal F$. Then any twisted extension $C=A(0\mid x_1,\dots,x_m)$ has the Koszul filtration $\mathcal H=\{I^e+(Y)\mid I\in\mathcal F, Y\subseteq\{x_1,\dots,x_m\}\}$.
\end{cor}
\begin{proof}
If $t=0$, \eqref{eq:good filtration} is automatically satisfied.
\end{proof}

\section{Unconditional results}\label{sec:unconditional}
\subsection{Abstract Witt rings}\label{ssec:witt}
Let $F$ be a field of $\chr F\neq2$. This is of course equivalent to $-1$ being a primitive second root of $1$ in $F$, as prescribed by the Standing Hypothesis \ref{ass}.
In his celebrated paper \cite{milnor}, J.~Milnor observed the existence of a deep connection between three central arithmetic objects: the associated graded ring of the Witt ring $W(F)$ of quadratic forms over $F$, the Galois cohomology $H^\bullet(G_F,\F_2)$ of the absolute Galois group of $F$, and the reduced Milnor K-theory $K_\bullet(F)/2K_\bullet(F)$. This connection takes the shape of homomorphisms of graded algebras
\[
\xymatrix{
& K_\bullet(F)/2K_\bullet(F)\ar[dl]^{\nu}\ar[dr]^{\eta} & \\
\gr W(F)\ar[rr]^{e} && H^\bullet(G_F,\F_2),
}
\]
which Milnor proved to be isomorphisms in some special circumstances. \emph{Milnor conjectures} claim that the above maps are isomorphisms of graded algebras for all fields of characteristic different from $2$.
Milnor conjectures have recently been proved in full generality: the map $\eta$ was shown to always be an isomorphism by Voevodsky in \cite{voe}, while the map $e$ was shown to always be well-defined and an isomorphism by D.~Orlov, Vishik and Voevodsky in \cite{orvivo}.
The precise connection between Galois groups and Witt rings is established in \cite{MSp}. Analogous results for odd $p$ and the graded Witt ring replaced by Galois cohomology are treated in \cite{efrmin:witt} and \cite{cem:quot}.

There have been several attempts to encode the abstract ring-theoretic properties of Witt rings of quadratic forms into a set of axioms. The aim was to define a unified class of rings that includes Witt rings of quadratic forms and at the same time is flexible enough to also describe other families of Witt rings. One of the two equivalent definitions that is now the most commonly used is the following (see \cite{marshall:absWitt}).
\begin{defn}
An \emph{abstract Witt ring} is a pair $W=(R,G)$ such that
\begin{enumerate}
\item $R$ is a commutative ring with unit $1$;
\item $G$ is a subgroup of the multiplicative group $R^\times$ (the \emph{group of square classes}), it contains $-1$, it has exponent $2$, and it generates $R$ as an additive group;
\item The ideal $I_W\id R$ generated by elements of the form $a+b$, with $a,b \in G$, (the \emph{augmentation ideal}) satisfies
\begin{enumerate}
\item[(AP1):] If $a \in G$, then $a \not\in I_W$;
\item[(AP2):] If $a,b \in G$ and $a+b \in I_W^2$, then $a+b=0$;
\item[(WC):] If $a_1+\cdots+a_n = b_1+\cdots+b_n$, with $n \geq 3$ and all $a_i,b_i \in G$, then there exist $a,b,c_3,\cdots,c_n \in G$ such that $a_2+\cdots+a_n = a + c_3\cdots+c_n$ and $a_1 + a = b_1 + b$.
\end{enumerate}
\end{enumerate}

A \emph{morphism} of abstract Witt rings $(R_1,G_1)\to (R_2,G_2)$ is a ring homomorphism $\alpha:R_1\to R_2$ such that $\alpha(G_1)\subseteq\alpha(G_2)$.
\end{defn}
\begin{rem}
The Witt ring of quadratic forms over a field $F$ can be viewed as an abstract Witt ring $W=(W(F),G)$ with $G=F^\times/(F^\times)^2$, and $I_W$ is the ideal of $W(F)$ of even-dimensional quadratic forms. Axioms (AP1) and (AP2) are the first two instances of the infinite family of \emph{Arason-Pfister properties}
\[
AP(k):\mbox{ if }a_1+\dots+a_n\in I_W^k \mbox{ and } n<2^k\mbox{, then }a_1+\dots+a_n=0.
\]
Axiom (WC) is a consequence of the Witt Chain Lemma and the (Dickson-)Witt Cancellation Theorem (\cite{dickson}, \cite{witt}).

Observe that, for any abstract Witt ring $W=(R,G)$, $W/I_W=\Z/2\Z$, as it is for traditional Witt rings of quadratic forms.
\end{rem}
In the category of abstract Witt rings one can perform two basic operations: the direct product and the group ring construction.
\begin{defn}
Let $W_1=(R_1,G_1)$ and $W_2=(R_2,G_2)$ be abstract Witt rings. The \emph{direct product} of $W_1$ and $W_2$ is the abstract Witt ring $W_1\pr W_2=(R,G)$, with $G=G_1\prod G_2$ being the direct product of $G_1$ and $G_2$, and $R$ being the subring of the ring-theoretic direct product $R_1\prod R_2$ that is additively generated by $G$.
\end{defn}
\begin{defn}\label{def:group ring}
The \emph{group ring} of the abstract Witt ring $W=(R,G)$ \emph{over the group} $C_2=\{1,x\}$ of order $2$ is $W[x]=(R[C_2], G\times C_2)$.
\end{defn}

To each abstract Witt ring we associate a graded object defined by successive powers of its augmentation ideal.
\begin{defn}\label{defn:graded Witt ring}
Let $W=(R,G)$ be an abstract Witt ring with augmentation ideal $I_W$. The associated \emph{graded abstract Witt ring} is
\[
\gr W=\oplus_{i=0}^\infty I_W^i/I_W^{i+1},
\]
where by convention $I_W^0=R$.
If $r\in I_W^{i}$, the corresponding element $\overline r=r+I_W^{i+1}\in I_W^{i}/I_W^{i+1}$ of $\gr W$ is called the \emph{initial form} of $r$.
\end{defn}
The product of $R$ induces a well-defined product on $\gr W$ in the usual way: for $r\in I_W^{i}$ and $s\in I_W^{j}$,
\[
\overline{r}\cdot\overline{s}=\overline{rs}=rs+I_W^{i+j+1}\in I_W^{i+j}/I_W^{i+j+1}.
\]
\begin{prop}
Let $W_1=(R_1,G_1)$, $W_2=(R_2,G_2)$ be abstract Witt rings. Then $\gr (W_1\pr W_2)=\gr W_1\sqcap\gr W_2$.
\end{prop}
\begin{proof}
The direct product is defined in such a way that, for $i\geq 1$, $I^i_{W_1\pr W_2}=I^i_{W_1}\prod I^i_{W_2}$, the ring-theoretic direct product of ideals (see \cite[Section 5.4]{marshall:absWitt}). 
In fact, if $a,b\in G_1$, then $(a+b,0)=(a,1)+(b,-1)$ is the sum of two elements in $G_1\prod G_2$, which proves that $I_{W_1}\prod\{0\}\subseteq I_{W_1\pr W_2}$. The inclusion $\{0\}\prod I_{W_2}\subseteq I_{W_1\pr W_2}$ is proven analogously. From these inclusions, it follows that, for all $i\geq 1$, $I^i_{W_1}\prod\{0\}\subseteq I^i_{W_1\pr W_2}$ and $\{0\}\prod I^i_{W_2}\subseteq I^i_{W_1\pr W_2}$, and hence $I^i_{W_1}\prod I^i_{W_2}\subseteq I^i_{W_1\pr W_2}$. The reverse inclusion follows easily from the definition of the operations of a ring-theoretic direct product.
Now, for any $i\geq 1$ the maps
\[\begin{array}{l}
\Phi_i:{I_{W_1}^i}/{I_{W_1}^{i+1}}\prod{I_{W_2}^i}/{I_{W_2}^{i+1}}\to{(I_{W_1}^i\prod I_{W_2}^i)}/{(I_{W_1}^{i+1}\prod I_{W_2}^{i+1})}\\
\\
\Phi_i(a+I_{W_1}^{i+1},b+I_{W_2}^{i+1})=(a,b)+(I_{W_1}^{i+1}\prod I_{W_2}^{i+1})
\end{array}
\]
are well-defined isomorphisms of $\F_2$-vector spaces. We also define $\Phi_0=id_{\F_2}$. These maps are also compatible with the product, in the sense that the diagram
\[
\xymatrix@C-24pt{
\left(\frac{I_{W_1}^i}{I_{W_1}^{i+1}}\prod\frac{I_{W_2}^i}{I_{W_2}^{i+1}}\right)\ar_{\Phi_i}[d] & \times & \left(\frac{I_{W_1}^j}{I_{W_1}^{j+1}}\prod\frac{I_{W_2}^j}{I_{W_2}^{j+1}}\right)\ar_{\Phi_j}[d]\ar^{\text{product in }\gr W_1\sqcap\gr W_2}[rrrrrrrrrrrrrrrr] &&&&&&&&&&&&&&&& \frac{I_{W_1}^{i+j}}{I_{W_1}^{i+j+1}}\prod\frac{I_{W_2}^{i+j}}{I_{W_2}^{i+j+1}}\ar_{\Phi_{i+j}}[d] \\
\frac{I_{W_1}^i\prod I_{W_2}^i}{I_{W_1}^{i+1}\prod I_{W_2}^{i+1}} & \times & \frac{I_{W_1}^j\prod I_{W_2}^j}{I_{W_1}^{j+1}\prod I_{W_2}^{j+1}}\ar^{\text{product in }\gr (W_1\pr W_2)}[rrrrrrrrrrrrrrrr] &&&&&&&&&&&&&&&& \frac{I_{W_1}^{i+j}\prod I_{W_2}^{i+j}}{I_{W_1}^{i+j+1}\prod I_{W_2}^{i+j+1}}
}
\]
commutes for any $i,j\geq 1$ and a similar compatibility relation holds for $\Phi_0$.
Hence $\Phi=\oplus_{i=0}^\infty\Phi_i:\gr W_1\sqcap\gr W_2\to\gr (W_1\pr W_2)$ is an isomorphism of graded algebras.
\end{proof}
\begin{prop}\label{prop:graded witt ring group construction}
Let $W=(R,G)$ be an abstract Witt ring. Then $\gr W[x]=(\gr W)(t\mid y)$, where $t=\overline{(1+1)}\in I_{W[x]}/I_{W[x]}^{2}$ and $y=\overline{(1+x)}\in I_{W[x]}/I_{W[x]}^{2}$.
\end{prop}
\begin{proof}
We begin noticing that, since $1+x\in I_{W[x]}$, $I^i_{W[x]}=I^i_{W}\prod (1+x)I^{i-1}_{W}$. Now, for any $i\geq 1$ the maps
\[\begin{array}{l}
\Phi_i:{I_{W}^i}/{I_{W}^{i+1}}\prod y{I_{W}^{i-1}}/{I_{W}^{i}}\to{(I_{W}^i\prod (1+x)I_{W}^{i-1})}/{(I_{W}^{i+1}\prod (1+x)I_{W}^{i})}\\
\\
\Phi_i(a+I_{W}^{i+1},y(b+I_{W}^{i}))=(a,(1+x)b)+(I_{W}^{i+1}\prod (1+x)I_{W}^{i})
\end{array}
\]
are well-defined isomorphisms of $\F_2$-vector spaces. We also define $\Phi_0=id_{\F_2}$.

Moreover, since $x$ has order $2$ by construction (see Definition \ref{def:group ring}), $(1+x)(1+x)=(1+1)(1+x)$ in $W[x]$, that is, $y^2=ty$ in $\gr W[x]$. As a consequence, the maps $\Phi_i$ are also compatible with the product, in the sense that the diagram
\[
\xymatrix@C-24pt{
\left(\frac{I_{W}^i}{I_{W}^{i+1}}\prod y\frac{I_{W}^{i-1}}{I_{W}^{i}}\right)\ar_{\Phi_i}[d] & \times & \left(\frac{I_{W}^{j}}{I_{W}^{j+1}}\prod y\frac{I_{W}^{j-1}}{I_{W}^{j}}\right)\ar_{\Phi_j}[d]\ar^{\text{product in }(\gr W_1)(t\mid y)}[rrrrrrrrrrrrrrrr] &&&&&&&&&&&&&&&& \frac{I_{W}^{i+j}}{I_{W}^{i+j+1}}\prod y\frac{I_{W}^{i+j-1}}{I_{W}^{i+j}}\ar_{\Phi_{i+j}}[d] \\
\frac{I_{W}^i\prod(1+x)I_{W}^{i-1}}{I_{W}^{i+1}\prod(1+x)I_{W}^{i}} & \times & \frac{I_{W}^j\prod(1+x)I_{W}^{j-1}}{I_{W}^{j+1}\prod(1+x)I_{W}^{j}}\ar^{\text{product in }\gr W[x]}[rrrrrrrrrrrrrrrr] &&&&&&&&&&&&&&&& \frac{I_{W}^{i+j}\prod(1+x)I_{W}^{i+j-1}}{I_{W}^{i+j+1}\prod(1+x)I_{W}^{i+j}}
}
\]
commutes for any $i,j\geq 1$ and a similar compatibility relation holds for $\Phi_0$. Hence $\Phi=\oplus_{i=0}^\infty\Phi_i:(\gr W)(t\mid y)\to\gr W[x]$ is an isomorphism of graded algebras.
\end{proof}

\begin{defn}
An abstract Witt ring $(R,G)$ is \emph{realisable} if there is a field $F$ of $\chr F\neq2$ such that $R\cong W(F)$ as rings and $G\cong F^\times/(F^\times)^2$ as groups. An abstract Witt ring $(R,G)$ is \emph{finitely generated} if $|G|<\infty$.
\end{defn}

The basic examples of realisable abstract Witt rings are % the trivial ring $W(\C)$ and 
the rings $W(F)$, where $F$ is a local field of characteristic 0, that is, $F$ is either $\R$, $\C$, or a finite extension of $\Q_p$ for some prime $p$. In addition, the direct product of two realisable abstract Witt rings is realisable (see \cite{kula1,kula2}). On the Galois-theoretic level, the direct product of realisable Witt rings corresponds to the free pro-$p$ product of pro-$p$ groups (see also \cite[Section 4.2]{Koszul1}). Finally, if $W$ is realisable as the Witt ring of quadratic forms over $F$, then also the group ring $W[x]$ is realisable as the Witt ring of quadratic forms over the field of formal power series $F((X))$. On the Galois-theoretic level, the group ring construction of a realisable Witt ring corresponds to the cyclotomic semidirect product of a pro-$p$ group with $\Z_p$ (see also \cite[Section 4.2]{Koszul1}).

\begin{defn}
An abstract Witt ring is said to be of \emph{elementary type} if it is obtained by applying finitely many direct products and group ring constructions, starting with the basic Witt rings $W(\mbox{local field})$.
\end{defn}
\begin{conj}[Elementary Type Conjectures for Witt rings]\label{conj:ETC Witt}\
\begin{description}
\item[{\rm Weak ETC}] Every realisable finitely generated abstract Witt ring is of elementary type.
\item[{\rm Strong ETC}] Every finitely generated abstract Witt ring is of elementary type (and hence realisable).
\end{description}
\end{conj}
Up to now, only very partial results in the direction of Elementary Type Conjectures for Witt rings are known. Notably,
\begin{thm}[{\cite[Section 5]{carmar}}]\label{fact:etc 32}
Strong ETC holds for abstract Witt rings $(R,G)$ with $|G|\leq 32$.
\end{thm}
From this and our results in Section \ref{sec:baddies} we obtain immediately the following unconditional result.
\begin{thm}\label{thm:32}
If $F$ is a field of $\chr F\neq 2$ and $|F^\times/(F^\times)^2|\leq 32$, then the algebra $H^\bullet(G_F(2),\F_2)$ has a Koszul filtration, and in particular it is Koszul.
\end{thm}
Similarly, in \cite{cordes} Witt rings of fields which contain at most four quaternion algebras are classified. In particular, they satisfy the Strong ETC, so we obtain the following.
\begin{thm}\label{thm:4quat}
If $F$ is a field of $\chr F\neq 2$ and containing at most $4$ quaternion algebras, then the algebra $H^\bullet(G_F(2),\F_2)$ has a Koszul filtration, and in particular it is Koszul.
\end{thm}

By Merkurjev's Theorem \cite{merk}, the number of quaternion algebras over a field $F$, of $\chr F\neq 2$, coincides with the cardinality $|H^2(G_F(2),\F_2)|$. Hence, the hypothesis of Theorem \ref{thm:4quat} can be equivalently expressed by saying that $\chr F\neq 2$ and $\dim_{\F_2}H^2(G_F(2),\F_2)\leq 2$. But this dimension is the minimal number of generating relations of $G_F(2)$ (see \cite[Section I.4.3]{serre}). Thus, Theorem \ref{thm:4quat} can be reformulated as follows.
\begin{thm}
If $F$ is a field of $\chr F\neq 2$ and $G_F(2)$ has a presentation with finitely many generators and at most $2$ generating relations, then $H^\bullet(G_F(2),\F_2)$ has a Koszul filtration, and in particular it is Koszul.
\end{thm}
This generalises the main result of \cite{claudio} in the case $p=2$.

\subsection{Pythagorean formally real fields}
A field $F$ of $\chr F\neq 2$ is \emph{Pythagorean} if $F^2+F^2=F^2$ and \emph{formally real} if $-1$ is not a sum of squares in $F$.
A form of Elementary Type Conjecture has been proved for the family $\mathcal{PFR}$ of Pythagorean formally real fields with finitely many square classes (see \cite{minac}, \cite{jacob1}). Namely, 
$\mathcal{PFR}$ contains \emph{Euclidean} fields, that are fields $F$ such that $(F^\times)^2+(F^\times)^2=(F^\times)^2$ and $F^\times=(F^\times)^2\cup-(F^\times)^2$ (see \cite{becker}) and 
\begin{thm}\label{thm:PFR}
The family of maximal pro-$2$ quotients of absolute Galois groups of fields in $\mathcal{PFR}$ is described inductively as follows.
\begin{enumerate}
\item For any Euclidean field $E\in\mathcal{PFR}$ (for instance, $E=\R$), $G_E(2)=\Z/2\Z$.
\item\label{op2} For any two fields $F_1,F_2\in\mathcal{PFR}$ there exists $F\in\mathcal{PFR}$ such that $G_{F_1}(2)\ast_2 G_{F_2}(2)\cong G_{F}(2)$.
\item\label{op3} For any field $F_0\in\mathcal{PFR}$ and any finite product $Z=\prod_{i=1}^m\Z_2$ there exists $F\in\mathcal{PFR}$ such that $Z\rtimes G_{F_0}(2)\cong G_{F}(2)$, where the action of $G_{F_0}(2)$ on $Z$ is given by
\[
\sigma^{-1}z\sigma=z^{-1}\qquad\mbox{for any }\sigma\in G_{F_0}(2)\setminus\{1\}, \sigma^2=1 \mbox{ and }z\in Z
\]
\end{enumerate}
Moreover, each maximal pro-$2$ quotient of the absolute Galois group of a field in $\mathcal{PFR}$ can be obtained inductively from Galois groups of Euclidean fields applying a finite sequence of the two operations described in (2) and (3).
\end{thm}
For any two fields $F_1,F_2\in\mathcal{PFR}$, $H^\bullet(G_{F_1}(2)\ast_2 G_{F_2}(2),\F_2)\cong H^\bullet(G_{F_1}(2),\F_2)\sqcap H^\bullet(G_{F_2}(2),\F_2)$ (see \cite[Theorem 4.1.4]{NSW}).
Also, for any field $F_0\in\mathcal{PFR}$ and any finite product $Z=\prod_{i=1}^m\Z_2$, $H^\bullet(Z\rtimes G_{F_0}(2),\F_2)=H^\bullet(G_{F_0}(2),\F_2)(t\mid x_1,\dots,x_m)$, with $t$ corresponding to the square class of $-1$ in Bloch-Kato isomorphism.
Since $H^\bullet(\Z/2\Z,\F_2)=\F_2[t]$ is universally Koszul, using Propositions \ref{prop:direct sum universally Koszul} and \ref{prop:twisted extension universally Koszul} we get the following.
\begin{thm}\label{thm:PFR uK}
Let $F$ be a Pythagorean formally real field. If $|F^\times/(F^\times)^2|<\infty$, then the algebra $H^\bullet(G_F(2),\F_2)$ is universally Koszul.
\end{thm}
\begin{rem}
To any locally finite-dimensional graded $K$-algebra $A$ we can associate the \emph{Hilbert series}
\[
h_A(z)=\sum_{n\in\N}(\dim_K A_n)z^n.
\]
If in addition $A$ is Koszul, then we can attach to it a sequence of \emph{Betti numbers} $\beta_n$, which measure the number of free summands in the component of degree $n$ of the Koszul complex of $A$ (see \cite[Sections 2.2, 2.3]{polpos} or \cite[Section 6.1]{eneHerzog}). The formal power series 
\[
p_A(z)=\sum_{n\in\N}\beta_nz^n
\]
is called the \emph{Poincar\'e series} of $A$. The two series are related by the formula
\begin{equation}\label{eq:hilbert-poincare'}
h_A(z)p_A(-z)=1.
\end{equation}

In \cite{minac:hilbseries}, for any Pythagorean formally real field $F$ with finitely many square classes, the Hilbert series of the algebra $H^\bullet(G_F(2),\F_2)$ is introduced under the name of the \emph{Poincar\'e series} of $G_F(2)$.
These series are then effectively determined from the structure of the space of orderings of $F$ via a simple recursive formula (\cite[Theorems 2 and 10]{minac:hilbseries}).
%Theorem 2 therein describes how to get these series from the structure of the space of orderings of $F$. This, along with Theorem 10 of the same paper, provides a simple recursive formula for all these Hilbert series, for any given number $2^n$ of square classes of the field $F$.
Using Theorem \ref{thm:PFR} and Equation \eqref{eq:hilbert-poincare'} above, these results in \cite{minac:hilbseries} can be extended immediately to derive an effective formula for the Poincar\'e series of the algebras $H^\bullet(G_F(2),\F_2)$.
\end{rem}

In \cite[Corollary 2.2]{jacwar2}, B.~Jacob and R.~Ware prove that if $F$ is a field of $\chr F\neq 2$, with $|F^\times/(F^\times)^2|<\infty$ and with an elementary type Witt ring of quadratic forms, then, for any $a\in F$, the Witt ring of quadratic forms of $F[\sqrt a]$ is of elementary type as well.

From basic Galois theory and the structure of $2$-groups, we obtain the result that, if $F$ is a field of $\chr F \neq 2$ and $K/F$ is any finite (not necessarily Galois) $2$-extension such that $K\subseteq F(2)$, then there is a finite tower of degree $2$ extensions $F\subset F_1\subset\dots\subset F_n\subset K$.
Therefore, by induction, we obtain a strengthening of Theorem \ref{thm:PFR uK}.
\begin{thm}
Let $F$ be Pythagorean formally real with finitely many square classes, and let $K/F$ be a finite $2$-extension such that $K\subseteq F(2)$. Then $H^\bullet(G_K(2),\F_2)$ is universally Koszul.
\end{thm}

\providecommand{\bysame}{\leavevmode\hbox to3em{\hrulefill}\thinspace}
\providecommand{\MR}{\relax\ifhmode\unskip\space\fi MR }
% \MRhref is called by the amsart/book/proc definition of \MR.
\providecommand{\MRhref}[2]{%
  \href{http://www.ams.org/mathscinet-getitem?mr=#1}{#2}
}
\providecommand{\href}[2]{#2}


\begin{thebibliography}{CDNRW13}

\bibitem[AS27a]{artsch1}
E.~Artin and O.~Schreier, \emph{Algebraische {K}onstruktion reeller
  {K}\"{o}rper}, Abh. Math. Sem. Univ. Hamburg \textbf{5} (1927), no.~1,
  85--99. Reprinted in: Artin, E., {\it Collected Papers} (S. Lang and J. Tate,
  eds.), Springer, 1965, 258--272.

\bibitem[AS27b]{artsch2}
\bysame, \emph{Eine {K}ennzeichnung der reell abgeschlossenen {K}\"{o}rper},
  Abh. Math. Sem. Univ. Hamburg \textbf{5} (1927), no.~1, 225--231. Reprinted
  in: Artin, E., {\it Collected Papers} (S. Lang and J. Tate, eds.), Springer,
  1965, 289--295.

\bibitem[Bec74]{becker}
E.~Becker, \emph{Euklidische {K}{\"o}rper und euklidische {H}{\"{u}}llen von
  {K}\"orpern}, J. Reine Angew. Math. \textbf{268/269} (1974), 41--52.
  A collection of articles dedicated to Helmut Hasse on his seventy-fifth
  birthday, II.
  
\bibitem[Ber78]{bergman}
G.~M.~Bergman, \emph{The diamond lemma for ring theory}, Adv. in Math. \textbf{29} (1978), no.~2, 178--218.

\bibitem[Bro12]{brown}
K.~S.~Brown, \emph{Cohomology of groups}, Graduate Texts in Mathematics, vol. 87, Springer-Verlag, New York, 2012.

\bibitem[BW98]{bucwin}
B.~Buchberger, F.~Winkler, \emph{Gr\"obner bases and applications}, London Mathematical Society Lecture Note Series, vol. 251, Cambridge University Press, Cambridge, UK, 1998.

\bibitem[CDNR13]{codero}
A.~Conca, E.~De~Negri and M.~E.~Rossi, \emph{{K}oszul algebras and regularity},
  Commutative Algebra: Expository Papers Dedicated to {D}avid {E}isenbud on the
  Occasion of His 65th Birthday (I.~Peeva, ed.), Springer, 2013, 285--315.

\bibitem[CEM12]{cem:quot}
S.~K.~Chebolu, I.~Efrat and J.~Min{\'a}{\v{c}}, \emph{Quotients of absolute
  {G}alois groups which determine the entire {G}alois cohomology}, Math. Ann.
  \textbf{352} (2012), no.~1, 205--221.

\bibitem[CM82]{carmar}
A.~B.~Carson and M.~Marshall, \emph{Decomposition of {W}itt rings}, Canad. J.
  Math. \textbf{34} (1982), no.~6, 1276--1302.

\bibitem[CMS16]{chmish}
S.~Chebolu, J.~Min\'{a}\v{c} and A.~Schultz, \emph{Galois
  {$p$}-groups and {G}alois modules}, Rocky Mountain J. Math. \textbf{46}
  (2016), no.~5, 1405--1446.

\bibitem[Con00]{conca:u-Koszul}
A.~Conca, \emph{Universally {K}oszul algebras}, Math. Ann. \textbf{317} (2000),
  no.~2, 329--346.

\bibitem[Cor82]{cordes}
C.~M.~Cordes, \emph{Quadratic forms over fields with four quaternion algebras},
  Acta Arith. \textbf{41} (1982), no.~1, 55--70.

\bibitem[CRV01]{corova}
A.~Conca, M.~E.~Rossi and G.~Valla, \emph{Gr\"{o}bner flags and {G}orenstein
  algebras}, Compositio Math. \textbf{129} (2001), no.~1, 95--121.

\bibitem[CTV01]{cotrva}
A.~Conca, N.~V.~Trung and G.~Valla, \emph{Koszul property for points in
  projective spaces}, Math. Scand. \textbf{89} (2001), no.~2, 201--216.

\bibitem[Dem61]{demushkin1961}
S.~P.~Demu\v{s}kin, \emph{The group of a maximal {$p$}-extension of a local
  field}, Izv. Akad. Nauk SSSR Ser. Mat. \textbf{25} (1961), 329--346.

\bibitem[Dem63]{demushkin1963}
\bysame, \emph{On {$2$}-extensions of a local field}, Sibirsk. Mat. \v Z.
  \textbf{4} (1963), 951--955.

\bibitem[Dic07]{dickson}
L.~E.~Dickson, \emph{On quadratic forms in a general field}, Bull. Amer. Math.
  Soc. \textbf{14} (1907), no.~3, 108--115.

\bibitem[Efr95]{ido:ETC}
I.~Efrat, \emph{Orderings, valuations, and free products of {G}alois groups},
  Sem. Structure Alg{\'e}briques Ordonn{\'e}es, {U}niv. {P}aris {VII} (1995).

\bibitem[EH12]{eneHerzog}
V.~Ene and J.~Herzog, \emph{Gr\"obner bases in commutative algebra}, Graduate
  Studies in Mathematics, vol. 130, American Mathematical Society, Providence,
  RI, 2012.

\bibitem[EHH15]{enhehi}
V.~Ene, J.~Herzog and T.~Hibi, \emph{Linear flags and {K}oszul filtrations},
  Kyoto J. Math. \textbf{55} (2015), no.~3, 517--530.

\bibitem[Eis95]{eisenbud}
D.~Eisenbud, \emph{Commutative algebra with a view toward algebraic geometry},
  Graduate Texts in Mathematics, vol. 150, Springer-Verlag, New York, 1995.

\bibitem[EL72]{elmanLam}
R.~Elman and T.~Y.~Lam, \emph{Quadratic forms over formally real fields and
  {P}ythagorean fields}, Amer. J. Math. \textbf{94} (1972), 1155--1194.

\bibitem[EM11]{efrmin:witt}
I.~Efrat and J.~Min\'{a}\v{c}, \emph{On the descending central sequence of
  absolute {G}alois groups}, Amer. J. Math. \textbf{133} (2011), no.~6,
  1503--1532.

\bibitem[EM17]{efrmat}
I.~Efrat and E.~Matzri, \emph{Triple {M}assey products and absolute {G}alois
  groups}, J. Eur. Math. Soc. (JEMS) \textbf{19} (2017), no.~12, 3629--3640.

\bibitem[EQ19]{efrqua}
I.~Efrat and C.~Quadrelli, \emph{The {K}ummerian property and maximal pro-{$p$}
  {G}alois groups}, J. Algebra \textbf{525} (2019), 284--310.

\bibitem[GM]{guimin}
P.~Guillot and J.~Min\'{a}\v{c}, \emph{Extensions of unipotent groups, {M}assey
  products and {G}alois cohomology}, Adv. Math. \textbf{354} (2019), 106748, 40 pp.

\bibitem[GMTW18]{gumito}
P.~Guillot, J.~Min{\'a}{\v{c}} and A.~Topaz, \emph{Four-fold {M}assey products
  in {G}alois cohomology. {W}ith an appendix by {O}.~{W}ittenberg}, Compos.
  Math. \textbf{154} (2018), no.~9, 1921--1959.

\bibitem[HHR00]{hehire}
J.~Herzog, T.~Hibi and G.~Restuccia, \emph{Strongly {K}oszul algebras}, Math.
  Scand. \textbf{86} (2000), no.~2, 161--178.

\bibitem[Hil90]{hilbert}
D.~Hilbert, \emph{{\"U}ber die {T}heorie der algebraischen {F}ormen}, Math.
  Ann. \textbf{36} (1890), no.~4, 473--534.

\bibitem[HW]{harwit}
  Y.~Harpaz and O.~Wittenberg, \emph{The Massey vanishing conjecture for number fields}, arXiv preprint arXiv:1904.06512.

\bibitem[HW09]{haeseweibel}
C.~Haesemeyer and C.~Weibel, \emph{Norm varieties and the chain lemma (after
  {M}arkus {R}ost)}, Algebraic topology, Abel Symp., vol.~4, Springer, 2009,
  95--130.

\bibitem[HW15]{hw}
M.~J.~Hopkins and K.~G.~Wickelgren, \emph{Splitting varieties for triple {M}assey
  products}, J. Pure Appl. Algebra \textbf{219} (2015), no.~5, 1304--1319.

\bibitem[Jac81]{jacob1}
B.~Jacob, \emph{On the structure of {P}ythagorean fields}, J. Algebra
  \textbf{68} (1981), no.~2, 247--267.

\bibitem[JW89]{jacwar}
B.~Jacob and R.~Ware, \emph{A recursive description of the maximal pro-{$2$}
  {G}alois group via {W}itt rings}, Math. Z. \textbf{200} (1989), no.~3,
  379--396.

\bibitem[JW91]{jacwar2}
\bysame, \emph{Realizing dyadic factors of elementary type {W}itt rings and
  pro-{$2$} {G}alois groups}, Math. Z. \textbf{208} (1991), no.~2, 193--208.

\bibitem[Kul79]{kula1}
M.~Kula, \emph{Fields with prescribed quadratic form schemes}, Math. Z.
  \textbf{167} (1979), no.~3, 201--212.

\bibitem[Kul85]{kula2}
\bysame, \emph{Fields and quadratic form schemes}, Ann. Math. Sil. (1985),
  no.~13, 7--22.

\bibitem[Lab67]{labute:demushkin}
J.~P.~Labute, \emph{Classification of {D}emushkin groups}, Canad. J. Math.
  \textbf{19} (1967), 106--132.

\bibitem[Lam05]{lam2005}
T.~Y.~Lam, \emph{Introduction to quadratic forms over fields}, Graduate Studies
  in Mathematics, vol.~67, American Mathematical Society, Providence, RI, 2005.

\bibitem[LV12]{lodval}
J.-L.~Loday and B.~Vallette, \emph{Algebraic operads}, Grundlehren der
  Mathematischen Wissenschaften [Fundamental Principles of Mathematical
  Sciences], vol. 346, Springer-Verlag, Heidelberg, 2012.
  
\bibitem[LvdD81]{lubvdd}
A. Lubotzky and L. van den Dries, \emph{Subgroups of free profinite groups and
  large subfields of $\widetilde Q$}, Israel J. Math. \textbf{39} (1981),
  no.~1--2, 25--45.

\bibitem[Mar80]{marshall:absWitt}
M.~Marshall, \emph{Abstract {W}itt rings}, Queen's Papers in Pure and Applied
  Mathematics, vol.~57, Queen's University, Kingston, ON, 1980.

\bibitem[Mat18]{matzri:massey}
E.~Matzri, \emph{Triple {M}assey products of weight {$(1,n,1)$} in {G}alois
  cohomology}, J. Algebra \textbf{499} (2018), 272--280.

\bibitem[Mer81]{merk}
A.~S.~Merkurjev, \emph{On the norm residue symbol of degree {$2$}}, Dokl. Akad.
  Nauk SSSR \textbf{261} (1981), no.~3, 542--547, English translation: Soviet
  Math. Dokl. {\bf 24} (1981), no. 3, 546--551 (1982).

\bibitem[Mic13]{mich2}
I.~M.~Michailov, \emph{Galois realizability of groups of orders {$p^5$} and
  {$p^6$}}, Cent. Eur. J. Math. \textbf{11} (2013), no.~5, 910--923.

\bibitem[Mil70]{milnor}
J.~Milnor, \emph{Algebraic {$K$}-theory and quadratic forms}, Invent. Math.
  \textbf{9} (1970), 318--344.

\bibitem[Min86]{minac}
J.~Min{\'a}{\v{c}}, \emph{{G}alois groups of some {$2$}-extensions of ordered
  fields}, C. R. Math. Rep. Acad. Sci. Canada \textbf{8} (1986), no.~2,
  103--108.

\bibitem[Min93]{minac:hilbseries}
J.~Min\'{a}\v{c}, \emph{Poincar\'{e} polynomials, stability indices and number
  of orderings. {I}}, Advances in Number Theory ({K}ingston, {ON}, 1991),
  Oxford Sci. Publ., Oxford Univ. Press, 1993, 515--528.

\bibitem[MPQT]{Koszul1}
J.~Min{\'a}{\v{c}}, F.~W.~Pasini, C.~Quadrelli and N.~D.~T{\^a}n, \emph{Koszul
  algebras and quadratic duals in {G}alois cohomology}, arXiv preprint
  arXiv:1808.01695.

\bibitem[MRT]{mirota}
J.~Min\'{a}\v{c}, M.~Rogelstad and N.~D.~T\^{a}n, \emph{Relations in the
  maximal pro-{$p$} quotients of absolute {G}alois groups}, Trans. Amer. Math. Soc. \textbf{373} (2020), 2499--2524.

\bibitem[MS96]{MSp}
J.~Min{\'a}{\v{c}} and M.~Spira, \emph{Witt rings and {G}alois groups}, Ann. of
  Math. (2) \textbf{144} (1996), no.~1, 35--60.

\bibitem[MT16]{mintan1}
J.~Min\'{a}\v{c} and N.~D.~T\^{a}n, \emph{Triple {M}assey products vanish over
  all fields}, J. Lond. Math. Soc. (2) \textbf{94} (2016), no.~3, 909--932.

\bibitem[MT17]{mintan2}
\bysame, \emph{Triple {M}assey products and {G}alois theory}, J. Eur. Math.
  Soc. (JEMS) \textbf{19} (2017), no.~1, 255--284.

\bibitem[MZ11]{mich1}
I.~M.~Michailov and N.~P.~Ziapkov, \emph{On realizability of {$p$}-groups
  as {G}alois groups}, Serdica Math. J. \textbf{37} (2011), no.~3, 173--210.

\bibitem[NSW08]{NSW}
J.~Neukirch, A.~Schmidt and K.~Wingberg, \emph{Cohomology of number fields},
  second ed., Grundlehren der Mathematischen Wissenschaften [Fundamental
  Principles of Mathematical Sciences], vol. 323, Springer-Verlag, Berlin,
  2008.

\bibitem[OVV07]{orvivo}
D.~Orlov, A.~Vishik and V.~Voevodsky, \emph{An exact sequence for
  {$K^M_\ast/2$} with applications to quadratic forms}, Ann. of Math. (2)
  \textbf{165} (2007), no.~1, 1--13.

\bibitem[Pio01]{piontk:hilb}
D.~I.~Piontkovski\u{\i}, \emph{On {H}ilbert series of {K}oszul algebras}, Funct.
  Anal. Appl. \textbf{35} (2001), no.~2, 133--137.

\bibitem[Pio05]{piontk}
\bysame, \emph{Koszul algebras and their ideals}, Funct. Anal.
  Appl. \textbf{39} (2005), no.~2, 120--130.
  
\bibitem[Pio]{pio:noncomm}
\bysame, \emph{Noncommutative Koszul filtrations}, arXiv preprint
arXiv:math/0301233.

\bibitem[Pos95]{posit:hilb}
L.~Positselski, \emph{The correspondence between {H}ilbert series of
  quadratically dual algebras does not imply their having the {K}oszul
  property}, Funct. Anal. Appl. \textbf{29} (1995), no.~3, 213--217.

\bibitem[Pos98]{posit:thesis}
\bysame, \emph{Koszul property and {B}ogomolov's conjecture}, Ph.D. thesis,
  Harvard University, 1998, Available at
  \texttt{https://conf.math.illinois.edu/K-theory/0296/}.

\bibitem[Pos05]{positselsky:koszul}
\bysame, \emph{Koszul property and {B}ogomolov's conjecture}, Int. Math. Res.
  Not. (2005), no.~31, 1901--1936.

\bibitem[Pos14]{positselsky:koszul2}
\bysame, \emph{Galois cohomology of a number field is {K}oszul}, J. Number
  Theory \textbf{145} (2014), 126--152.

\bibitem[PP05]{polpos}
A.~Polishchuk and L.~Positselski, \emph{Quadratic algebras}, University Lecture
  Series, vol.~37, American Mathematical Society, Providence, RI, 2005.

\bibitem[Pri70]{priddy}
S.~B.~Priddy, \emph{Koszul resolutions}, Trans. Amer. Math. Soc. \textbf{152}
  (1970), 39--60.

\bibitem[PV95]{posivis:koszul}
L.~Positselski and A.~Vishik, \emph{Koszul duality and {G}alois cohomology},
  Math. Res. Lett. \textbf{2} (1995), no.~6, 771--781.

\bibitem[Qua]{claudio}
C.~Quadrelli, \emph{One relator maximal pro-{$p$} {G}alois groups and {K}oszul
  algebras}, arXiv preprint arXiv:1601.04480.

\bibitem[Roo95]{roos}
J.-E.~Roos, \emph{On the characterisation of {K}oszul algebras. {F}our
  counterexamples}, C. R. Acad. Sci. Paris S\'{e}r. I Math. \textbf{321}
  (1995), no.~1, 15--20.

\bibitem[Sch14]{andy}
A.~Schultz, \emph{Parameterizing solutions to any {G}alois embedding
  problem over {$\mathbb{Z}/p^n\mathbb{Z}$} with elementary {$p$}-abelian
  kernel}, J. Algebra \textbf{411} (2014), 50--91.

\bibitem[Ser95]{serre:demushkin}
J.-P.~Serre, \emph{Structure de certains pro-{$p$}-groupes (d'apr\`es {D}emu\v
  skin)}, S\'eminaire {B}ourbaki, {V}ol.\ 8, Soc. Math. France, Paris, 1995,
  Exp.~No.~252, 145--155.

\bibitem[Ser02]{serre}
\bysame, \emph{Galois cohomology}, English ed., Springer Monographs in
  Mathematics, Springer-Verlag, Berlin, 2002.

\bibitem[SJ06]{suslin:norm}
A.~Suslin and S.~Joukhovitski, \emph{Norm varieties}, J. Pure Appl. Algebra
  \textbf{206} (2006), no.~1-2, 245--276.

\bibitem[Voe03]{voe}
V.~Voevodsky, \emph{Motivic cohomology with {${\mathbf Z}/2$}-coefficients},
  Publ. Math. Inst. Hautes \'Etudes Sci. (2003), no.~98, 59--104.

\bibitem[Voe10]{voevodsky:motivic}
\bysame, \emph{Motivic {E}ilenberg-{M}acLane spaces}, Publ. Math. Inst. Hautes
  \'Etudes Sci. (2010), no.~112, 1--99.

\bibitem[Voe11]{voevodsky:BK}
\bysame, \emph{On motivic cohomology with {$\mathbf Z/l$}-coefficients}, Ann.
  of Math. (2) \textbf{174} (2011), no.~1, 401--438.

\bibitem[Wad83]{wadsworth:cohomology}
A.~R.~Wadsworth, \emph{{$p$}-{H}enselian fields: {$K$}-theory, {G}alois
  cohomology, and graded {W}itt rings}, Pacific J. Math. \textbf{105} (1983),
  no.~2, 473--496.

\bibitem[War78]{ware}
R.~Ware, \emph{When are {W}itt rings group rings? {II}}, Pacific J. Math.
  \textbf{76} (1978), no.~2, 541--564.

\bibitem[Wic12]{wick1}
K.~Wickelgren, \emph{On {$3$}-nilpotent obstructions to {$\pi_1$} sections
  for {$\mathbb {P}^1_\mathbb Q-\{0,1,\infty\}$}}, The arithmetic of
  fundamental groups---{PIA} 2010, Contrib. Math. Comput. Sci., vol.~2,
  Springer, Heidelberg, 2012, pp.~281--328.

\bibitem[Wic17]{wick2}
\bysame, \emph{Massey products {$\langle y$}, {$x$}, {$x$}, {$\ldots$}, {$x$},
  {$x$}, {$y\rangle$} in {G}alois cohomology via rational points}, J. Pure
  Appl. Algebra \textbf{221} (2017), no.~7, 1845--1866.

\bibitem[Wit37]{witt}
E.~Witt, \emph{Theorie der quadratischen {F}ormen in beliebigen {K}\"orpern},
  J. Reine Angew. Math. \textbf{176} (1937), 31--44.

\end{thebibliography}
\end{document}